\documentclass[reqno]{amsart}

\usepackage{amsmath,amssymb,color}
\usepackage{amsfonts, amscd, epsfig, amsmath, amssymb,enumerate}
\usepackage{graphicx}
\usepackage{color}
\usepackage{mathrsfs}
\usepackage{mathptmx}       
\usepackage{helvet}         
\usepackage{courier}        

\usepackage{tikz}
\usetikzlibrary{backgrounds}
\usetikzlibrary{patterns,fadings}
\usetikzlibrary{arrows,decorations.pathmorphing}
\usetikzlibrary{decorations}
\usetikzlibrary{calc}
\definecolor{light-gray}{gray}{0.95}
\usepackage{float}
\usepackage[colorlinks=true,linkcolor=blue,citecolor=magenta]{hyperref}
\def\centerarc[#1](#2)(#3:#4:#5){\draw[#1] ($(#2)+({#5*cos(#3)},{#5*sin(#3)})$) arc (#3:#4:#5);}

\newcommand*\circled[1]{\tikz[baseline=(char.base)]{
            \node[shape=circle,draw,inner sep=1pt] (char) {#1};}}

\newtheorem{theorem}{Theorem}[section]
\newtheorem{lemma}[theorem]{Lemma}
\newtheorem{proposition}[theorem]{Proposition}
\newtheorem{corollary}[theorem]{Corollary}
\newtheorem{remark}[theorem]{Remark}
\newtheorem{definition}{Definition}[section]

\renewcommand\vec[1]{\overrightarrow{#1}}
\newcommand\vecleft[1]{\overleftarrow{#1}}

\numberwithin{equation}{section}

\newcommand{\mc}[1]{{\mathcal #1}}

\newcommand{\bb}[1]{{\mathbb #1}}

\renewcommand{\epsilon}{\varepsilon}

\newcommand{\R}{\mathbb R}

\newcommand{\Z}{\mathbb Z}
\newcommand{\N}{\mathbb N}
\renewcommand{\P}{\mathbb P}

\newcommand{\E}{\mathbb E}

\title[Convergence to the  SBE from a degenerate dynamics]{Convergence to the stochastic Burgers equation\\ from a  degenerate microscopic dynamics}

\author{Oriane Blondel}
\address{\noindent CNRS, Univ Lyon, Universit\'e Claude Bernard Lyon 1, CNRS UMR 5208, Institut Camille Jordan,
43 boulevard du 11 novembre 1918 – 69622, France}
\email{blondel@math.univ-lyon1.fr}

\author{Patr\'{\i}cia Gon\c{c}alves}
\address{\noindent Center for Mathematical Analysis,  Geometry and Dynamical Systems,
Instituto Superior T\'ecnico, Universidade de Lisboa,
Av. Rovisco Pais, 1049-001 Lisboa, Portugal.}
\email{patricia.goncalves@math.tecnico.ulisboa.pt}

\author{Marielle Simon}
\address{\noindent \'Equipe MEPHYSTO, Inria Lille -- Nord Europe, 40 avenue du Halley,
59650 Villeneuve d'Ascq, France. }\email{marielle.simon@inria.fr}

\begin{document}
\begin{abstract}
In this paper we prove the convergence to the stochastic Burgers equation from one-dimensional interacting particle systems, whose dynamics allow the degeneracy of the jump rates. To this aim, we provide a  new proof of the second order Boltzmann-Gibbs principle introduced in \cite{gj2014}. The main technical difficulty  is that our models exhibit configurations that do not evolve under the dynamics - the \emph{blocked configurations} - and are locally non-ergodic. Our proof does not impose any knowledge on the  spectral gap for the microscopic models. Instead, it  relies on the fact that, under the equilibrium measure, the probability to find a blocked configuration in a finite box is exponentially small in the size of the box. Then, a dynamical mechanism allows to exchange particles even when the jump rate for the direct exchange is zero.
\end{abstract}

\maketitle


\section{Introduction}
In the last few years there has been an intense research activity around the  derivation of the stochastic Burgers equation (SBE), or its integrated counterpart, namely the Kardar-Parisi-Zhang (KPZ) equation, from one-dimensional weakly asymmetric conservative interacting particle systems. The KPZ equation goes back to  \cite{KPZ} where it has been proposed as the default stochastic partial differential equation (SPDE) ruling  the evolution of the profile of a randomly growing interface. If $h(t,x)$ denotes the height of the interface at time $t$ and position $x$ , then the KPZ equation reads as \[dh(t,x) = A\Delta h(t,x)\; dt + B(\nabla h(t,x))^2 \; dt + \sqrt C \; {d\mc W_t},\]
where $A$, $B$ and $C$ are constants  depending on the thermodynamical quantities of the system (see \cite{GJ4} for instance), $\Delta$ and $\nabla$ are, respectively, the Laplacian and derivative operators 
and $ \mc W_t$ is a {Brownian motion}. The SBE can  be obtained from the KPZ equation, at least formally, by taking the space derivative of $h(t,x)$, namely, $\mathcal{Y}_t
= \nabla h_t,$ so that $\mathcal{Y}_t$ solves the SBE given by
\[d\mathcal{Y}_t = A \Delta\mathcal{Y}_t \; dt + B \nabla(\mathcal{Y}_t)^2 \; dt + \sqrt C  \; \nabla d\mc W_t.\]
 Mathematically, the KPZ equation is ill-posed, the problematic term being  $(\nabla h(t,x))^2$, which is not well-defined if $h$ has the regularity we expect for a solution. A  way to solve this equation is to consider its Cole-Hopf solution, which solves a 
stochastic heat equation with a multiplicative noise. Since the latter equation is linear, its solutions can be constructed and uniqueness as well as existence of solutions can be easily obtained. The problem of uniqueness at the level of the KPZ equation is much more complicated and it has not been proved yet that the Cole-Hopf solutions do, in fact, solve the KPZ equation in some reasonable way.  
In the last couple of years, Hairer has developed a meaningful notion of solutions for the KPZ equation
and proved uniqueness of such solutions with periodic boundary conditions, see \cite{Hai,Hai2}. However, as far as we know, it is still  not known how to use this notion to prove a meaningful convergence for our kind of microscopic systems to the equation, and it is not what we will be using here.

In the seminal paper \cite{BerG}, it is shown that Cole-Hopf solutions of the SBE equation can be obtained as a scaling
limit of  the weakly asymmetric simple exclusion process. The approach there consists in performing the Cole-Hopf transformation at the microscopic level, which had been already introduced in \cite{Gar,DG}, in order to obtain, respectively, the propagation of chaos and  the non-equilibrium fluctuations of the same model. The microscopic Cole-Hopf transformation, similarly to what happens at the macroscopic level, solves a linear
equation and the problem of the non-linearity is completely avoided under this transformation. 
The drawback of this method is that it relies on the possibility to perform a microscopic 
Cole-Hopf transformation, which cannot be done for most models.
A way to overcome this difficulty has been proposed in \cite{gj2014}, where the results of \cite{BerG} have been extended to general weakly asymmetric exclusion processes without requiring the possibility of a microscopic Cole-Hopf transformation. The key ingredient in the authors' approach is the derivation of a \emph{second order Boltzmann-Gibbs principle} (BGP2), which allows to make the non-linear term of the SBE equation emerge from the underlying microscopic dynamics. Technically, this principle is obtained by a multi-scale argument which consists in replacing a local function by another function whose support increases at each step and whose variance decreases in order to make the errors eventually vanish. This argument had already been used in \cite{G,G1} to show the triviality of the fluctuations of the asymmetric simple exclusion and the asymmetric zero-range process. The BGP2  is shown to hold for a large class of weakly asymmetric exclusion processes in \cite{gj2014,GJS}, and for zero-range, non-degenerate kinetically constrained models and other models starting from a state close to equilibrium in \cite{GJS}. We note that none of 
the aforementioned models admits a microscopic Cole-Hopf transformation and therefore the methods of \cite{BerG}
could not be used there.

In \cite{gj2014}, a notion of solution to the SBE appears as a natural consequence of the estimates obtained through the BGP2: it is called \emph{energy solution}. More precisely, the authors write the martingale problem for the microscopic dynamics and perform the replacements allowed by the BGP2. Then they find that a microscopic version of the SBE is satisfied, and therefore any limit has to be a solution of the macroscopic SBE, which also satisfies certain energy estimates that come from the BGP2. In order to complete the picture, it remained until very recently to prove uniqueness of such energy solutions. {This happened in \cite{GubPer}, where the authors prove uniqueness, but using the notion of stationary controlled solutions of the SBE which had  been stated previously in \cite{GubJar} and which is, apparently, a bit more restrictive than the notion of energy solution introduced in \cite{gj2014,GJS}.} We note however that, it is quite easy to check that the extra conditions in the notion of  stationary controlled solutions of \cite{GubJar,GubPer} are also satisfied for the dynamics in \cite{gj2014,GJS}, see, for example, \cite{GS} and references therein. As a consequence of that, in \cite{GS} it is introduced a notion of stationary energy solutions of the SBE which is equivalent to the notion of stationary controlled solutions of \cite{GubJar,GubPer}.  We prefered to opt for this notion of solutions since the required conditions arise very naturally from the underlying microscopic models that we consider. So that is why here we will use this stronger notion of stationary energy solutions  of \cite{GS} to show convergence to the SBE, which will then be a consequence of the second order Boltzmann-Gibbs principle.
  To sum up, the methods of \cite{gj2014,GJS,GS} have then opened a new way to obtain the convergence to the  SBE from general interacting particle systems which satisfy some constraints at the level of the rates. 

 Let us recap under which conditions the BGP2 has been proved so far.
  In \cite{gj2014} it is assumed that the models satisfy a sharp  bound on the spectral gap, which is uniform on the density of particles, and that the rates are bounded from below and above by a positive constant. In \cite{GJS} it is assumed that the models satisfy a spectral gap bound which does not need to be uniform on the density of particles and the jump rates can vanish in the scaling limit. More recently in \cite{FGS,GS} a new proof of the BGP2 has been introduced: it permits to extend the previous results to models which do not need to satisfy a spectral gap bound, as, for example, the symmetric exclusion with a slow bond \cite{FGS}. In all the aforementioned works it is assumed that the symmetric version of the models is gradient, that is, its instantaneous current is written as the gradient of some local function. Nevertheless,  the possible degeneracy of the rates is admitted in none of these works. 
We note that in \cite{GLT} (respectively \cite{Nag}), for the collection of models that we consider and  which have degenerate rates, a lower bound estimate was obtained for the spectral gap, on a regime of high (respectively low) densities. We note, however, that even combining the two mentioned estimates, it is not possible to use the approach of \cite{gj2014} or \cite{GJS} to obtain our results, since in both works the 
non-degeneracy of the rates is crucial.

In this article we give a step towards generalizing the previous results, since our main contribution is that we prove the convergence to the  SBE from  microscopic dynamics that allow degenerate exchange rates. This goes towards showing that the more recent approach  of \cite{FGS,GS} to proving BGP2 is more flexible to situations which lack uniformity. 
We consider kinetically constrained lattice gases as in \cite{GJS,GLT}, but under the presence of states which do not evolve under the dynamics. These models have been introduced  in the 
physical literature since the eighties. In the case where the rates are symmetric the hydrodynamic limit has been derived in \cite{GLT} (under restrictions on the initial density profile) and it is given by the \emph{porous medium equation}. The  dynamics of our models can be described as follows. Fix an integer $m\geq 2$.  The process of configurations $\eta$ evolves on $\{0,1\}^\mathbb{Z}$ and we say there is a particle at $x$ (resp. $x$ is empty) if $\eta(x)=1$ (resp. $\eta(x)=0$). If there is a particle at the site $x \in \bb Z$ and, and if $x+1$ (resp. $x-1$) is unoccupied, then it jumps to $x+1$ (resp. $x-1$) at rate $p(1)c^m_{x,x+1}(\eta)$ (resp. $p(-1)c^m_{x,x-1}(\eta)$). Here, $c^m_{x,x+1}(\eta)=c^m_{x+1,x}(\eta)$ represents a constraint: the jump takes place only if there are at least $m-1$ particles around $x$, see \eqref{eq:rate} for a precise definition. We consider the weakly asymmetric version of the model of \cite{GLT}, so that above $p(\pm 1)=\frac12\pm \frac{b}{2n^\gamma}$, where $b\in\bb R$, $\gamma\geq \frac12$ and $n\in\bb N$ is a scaling parameter. The last parameter $\gamma$ rules the strength of the asymmetry in the model and $b$ is just a constant that we can choose as being non-zero to turn the model weakly asymmetric. As hinted above, we prove the BGP2 for this model. As a consequence, we obtain a crossover from the Edwards-Wilkinson universality class to the KPZ universality class when the strength of the asymmetry $\gamma$ varies, at the level of density fluctuations. More precisely, we show that for $\gamma>\frac12$ the equilibrium density fluctuations are Gaussian and given by a generalized Ornstein-Uhlenbeck process and for $\gamma=\frac12$ they are non-Gaussian and given by the energy solution of the  SBE. 

We now comment the distinctive features of our dynamics. The first one is that the symmetric version of the model ($b=0$) is gradient, see \eqref{grad_cond_sym}. The second is that the jump rates are zero if the constraint is not satisfied. In particular, there are configurations that never evolve under the dynamics (\emph{blocked configurations}): those where no group of $m-1$ particles is at distance two or less from another particle. As a consequence, we have several invariant measures: the Bernoulli product measure with a constant parameter $\rho\in{[0,1]}$ that we denote by $\nu_\rho$, but also any Dirac measure supported on a blocked configuration. In the following, we consider $\rho\in{(0,1)}$ and $\nu_\rho$ will be our reference measure. Fortunately, almost surely under $\nu_\rho$, there is no blocked configuration. The most important ingredient for us is the existence of a \emph{mobile cluster} inside finite boxes, which occurs with very high probability, see \eqref{eq:decaysigmam}. This mobile cluster is a collection of $m$ neighbouring particles, which can move through the system on their own (if the $m$ particles keep together, the jump constraints are always satisfied) and can be used as a vehicle to transport a particle from a site $x$ to another site $y$. This mechanism (see Section \ref{sec:blocked} for a more precise description) {has been introduced in \cite{GLT}} and is going to help us overcome the degeneracy of the rates.

Here follows an outline of this paper. In Section \ref{sec:replacement} we introduce notation and we state all the properties we impose on the microscopic models. In Section \ref{sec:main_results} we state the notion of energy solutions of the SBE equation, for which uniqueness has been proved and we also state the crossover from the Ornstein-Uhlenbeck process to the energy solution of the  SBE. In Section \ref{sec:BG} we prove the BGP2 and in Section \ref{sec:corollary} we derive some simple consequences of this principle.

\section{The microscopic dynamics}\label{sec:replacement}

\subsection{Description of the model}

Fix an integer $m\geq 2$. We consider  a collection of microscopic dynamics of exclusion type, which may be thought of as microscopic models for the porous medium equation given in \eqref{eq:porous}  and were studied  in \cite{GLT} and references therein. These are Markov processes that we denote by $\{\eta_t^n\,;\, t\geq 0\}$, which have  state space  $\mc E: = \{0,1\}^\Z$ and  whose infinitesimal generator $\mathcal{L}_n^m$ is defined below. For any $\eta \in \mc E$ and $x\in\bb Z$, we represent by $\eta(x)$  the occupation variable at site $x$. 

Hereafter and above, $n\in\bb N$ is a scaling parameter that shall go to infinity. In order to define the transition kernel of the dynamics, we need two real parameters $\gamma, b$  such that
$\gamma\geq \frac12$. For any $x,y\in\bb Z$, we denote by $\eta^{x,y}$ the configuration obtained from $\eta$ after exchanging the occupation values $\eta(x)$ and $\eta(y)$ as follows:
\[(\eta^{x,y})(z)=\begin{cases} \eta(x) & \text{if }z=y, \\
\eta(y) & \text{if }z=x, \\
\eta(z) & \text{otherwise}.
\end{cases} \]
We say that a function $f:\mc E\to \R$ is \emph{local} if it depends on $\eta \in \mc E$ only through a finite number of coordinates.  
 With these notations,  $\mathcal{L}_n^m$ acts on local functions $f:\mc E\to \R$ as 
\begin{equation*}
(\mathcal{L}_n^mf)(\eta)=\sum_{x,y\in\mathbb{Z}\atop |x-y|=1}c_{x,y}^m(\eta)p(y-x)\eta(x)(1-\eta(y))\big(f(\eta^{x,y})-f(\eta)\big),
\end{equation*}
where $ p(x)=0$ for any $ |x|\neq 1$,  $p(\pm 1) =\frac12\pm \frac{b}{2n^\gamma},$
and $c_{x,y}^m$ represent constraints that are given by \begin{equation} 
c_{x,x+1}^m(\eta)=c_{x+1,x}^m(\eta)=\sum_{k=1}^m \prod_{\substack{j=-(m-k)\\ j \neq 0,1}}^{k} \eta(x+j). \label{eq:rate}
\end{equation}
In other words, each  particle waits, independently of the others, a random time exponentially distributed, and  jumps to  unoccupied nearest neighbour sites when at least $m-1\geq 1$ neighbouring sites are full.   We note that for $m=2$, the constraint is just $ 
c_{x,x+1}^{m}(\eta)=\eta(x-1)+\eta(x+2),$
 and for simplicity of notation we shall focus the presentation of our result for this choice of $m$. Whenever $m=2$, we drop it from the notation and along the paper we will explain how to adapt all the arguments to any integer value of $m$.

For $b=0$, which corresponds to the symmetric version of the model,  the \emph{hydrodynamic limit}\footnote{We refer to \cite{kl, Spohn} for reviews on the subject.} corresponds to the one-dimensional porous medium equation given by  \begin{equation} \label{eq:porous}\partial_t\rho(t,u)=\Delta(\rho^m)(t,u), \qquad (t,u) \in \bb R_+\times \bb R,\end{equation} where $\Delta$ represents the usual Laplacian operator.  
This has been proved in \cite{GLT}, in the case where the initial density profile is uniformly bounded away from $0$ and $1$. For more general initial profiles, the authors need to perturb the model, by adding jumps of the symmetric simple exclusion which allows to get rid of the degeneracy of the system. In the following we denote by $D(\rho):= m \rho^{m-1}$ the \emph{diffusion coefficient} associated with the porous medium equation given by  \eqref{eq:porous}.

\subsection{Invariant measures}

Due to the presence of constraints in the jump rates, the  system is not ergodic, since it admits  an infinite number of invariant measures: in particular, they contain the Dirac measures supported on configurations which cannot evolve under the dynamics (more details are exposed in Section \ref{sec:blocked} below).  

Consider now  $\rho\in(0,1)$, and let $\nu_\rho$ be the Bernoulli product measure on $\mc E$ with marginal at site $x$ given by
$
\nu_\rho\{\eta \; ; \; \eta(x)=1\}=\rho.
$
It is not difficult to see that $\nu_\rho$ is invariant for our model: for any local function $f:\mc E\to\bb R$, we have
$ \int_{\mc E} (\mc L_n^m f)(\eta)\; \nu_\rho(d\eta)=0.$ 
Let $\mathcal{D}(\R_+,\mc E)$ be the Skorohod space, that is, the space of $\mc E$-valued  trajectories which are continuous from the right with left limits and endowed with the weak topology. Let $\P_{\rho}^n$ be the probability measure on $\mathcal{D}(\R_+,\mc E)$ induced by the initial state $\nu_\rho$ and $\{\eta^n_{tn^2}\; ; \; t\geq 0\}$ be the accelerated Markov process  whose time evolution is described by the generator $n^2 \mathcal{L}_n^m$. The corresponding expectation is denoted by $\E^n_{\rho}$.

Let us fix once and for all $\rho \in (0,1)$, and introduce some notations: for any measurable $f:\mc E\to\bb R$, let $\nu_\rho(f)$ be the average of $f$ with respect to the measure $\nu_\rho$. If $f=\mathbf{1}_A$ is an indicator function of a Borel set $A\subset \mc E$, we simply denote $\nu_\rho(A):=\nu_\rho(\mathbf{1}_A)$. For any $x\in\bb Z$, we denote by $\tau_x f$ the  function obtained by translation as follows: $\tau_x f(\eta):=f(\tau_x \eta),$ where $(\tau_x\eta)(y) = \eta(x+y),$ for $y\in\bb Z.$
We also denote by $\bar{f}$ the \emph{recentered} function defined as
$
\bar{f}(\eta):=f(\eta)-\nu_\rho(f).
$ 
Finally, we  define $\chi(\rho)$ as the static compressibility of the system, namely $\chi(\rho):=\nu_\rho((\eta(0)-\rho)^2)=\rho(1-\rho)$.

\subsection{Blocked configurations} \label{sec:blocked}
In this section we describe more precisely the mechanisms at play in the dynamics between particles. For simplicity, we restrict the presentation to $m=2$.

Note  first that any configuration $\eta$ in which the distance between any two occupied sites is bigger than $2$ satisfies the following property: all the exchange rates $\eta(x)(1-\eta(y))c_{x,y}(\eta)$ are equal to zero. Such a configuration is called a \emph{blocked configuration} and a Dirac measure supported on it is an invariant measure for this process. 

If the configuration contains somewhere a pair of particles at distance at most $2$, we call them a \emph{mobile cluster}. It has the property that there exists a sequence of nearest neighbour jumps with positive rates (which we call \emph{allowed path}) which permits to shift the mobile cluster to any other position on $\bb Z$. As noted in \cite{GLT}, another crucial property of this mobile cluster is that, for any $x,y$ that do not belong to the cluster, there exists an {allowed path} that transports the cluster to the vicinity of $x,y$ and it permits to exchange $\eta(x)$ with $\eta(y)$ while restoring the other occupation variables to their initial value (in other words, a sequence of transitions with positive rates that turn $\eta$ into $\eta^{ x,y}$). Moreover, we can choose this allowed path in such a way that no bond is used more than six times. For a general $m\geq 3$ we have a very similar situation to the one above, but now  we need a group of $m$ particles with at most one hole inside the group in order to make the exchange $\eta(x)$ into $\eta(y)$, and on the allowed path  we use at most $2(m+1)$ times the same bond in order to make the exchange. We refer to Figure \ref{fig:allowedmoves} for an illustration in the case $m=2$.

\begin{figure}
\centering
\begin{tikzpicture}
\draw (-3,0) -- (3.5,0);
\foreach \a in {-2.5,-2,...,3}
   \draw[thick] (\a,-0.1) -- (\a,0.1);

\draw (-2,-0.3) node[anchor=north] {$x$};
\draw (2.5,-0.3) node[anchor=north] {$y$};

\filldraw[black, thick] (-2,0) circle (4pt) ;
\filldraw[black, thick] (-1,0) circle (4pt) ;
\filldraw[black, thick] (0,0) circle (4pt) ;
\filldraw[black, thick, fill=white] (2.5,0) circle (4pt) ;
\draw[color=violet, thick] (-1.3,-0.2) -- (-1.3,-0.4) -- (0.3,-0.4) -- (0.3,-0.2);
\draw[color=violet] (-0.5,-0.4) node[anchor=north] {\small mobile cluster};

\draw (5.2,0.3) node[anchor=north] {\emph{initial configuration}};


\draw (-3,-2) -- (3.5,-2);
\foreach \a in {-2.5,-2,...,3}
   \draw[thick] (\a,-2.1) -- (\a,-1.9);

\draw (-2,-2.3) node[anchor=north] {$x$};
\draw (2.5,-2.3) node[anchor=north] {$y$};

\filldraw[black, thick] (-2,-2) circle (4pt) ;
\filldraw[black, thick] (-1,-2) circle (4pt) ;
\filldraw[black, thick] (0,-2) circle (4pt) ;
\filldraw[black, thick, fill=white] (2.5,-2) circle (4pt) ;

\centerarc[thick,<-,color=violet](-1.75,-1.8)(0.3:180:0.3);

\draw[color=violet] (-1.75,-1.55) node[anchor=south] {\circled{\bf 1}};

\draw (5,-1.7) node[anchor=north] {\emph{allowed moves}};


\draw (-3,-3.5) -- (3.5,-3.5);
\foreach \a in {-2.5,-2,...,3}
   \draw[thick] (\a,-3.6) -- (\a,-3.4);

\draw (-2,-3.8) node[anchor=north] {$x$};
\draw (2.5,-3.8) node[anchor=north] {$y$};

\filldraw[black, thick] (-1.5,-3.5) circle (4pt) ;
\filldraw[black, thick] (0,-3.5) circle (4pt) ;
\filldraw[black, thick] (-1,-3.5) circle (4pt) ;
\filldraw[black, thick, fill=white] (2.5,-3.5) circle (4pt) ;

\centerarc[thick,<-,color=violet](-0.75,-3.3)(0.3:180:0.3);

\draw[color=violet] (-0.75,-3.05) node[anchor=south] {\circled{\bf 2}};


\draw (-3,-5) -- (3.5,-5);
\foreach \a in {-2.5,-2,...,3}
   \draw[thick] (\a,-5.1) -- (\a,-4.9);

\draw (-2,-5.3) node[anchor=north] {$x$};
\draw (2.5,-5.3) node[anchor=north] {$y$};

\filldraw[black, thick] (-1.5,-5) circle (4pt) ;
\filldraw[black, thick] (0,-5) circle (4pt) ;
\filldraw[black, thick] (-0.5,-5) circle (4pt) ;
\filldraw[black, thick, fill=white] (2.5,-5) circle (4pt) ;

\centerarc[thick,<-,color=violet](-1.25,-4.8)(0.3:180:0.3);

\draw[color=violet] (-1.25,-4.55) node[anchor=south] {\circled{\bf 3}};

\centerarc[thick,<-,color=violet](0.25,-4.8)(0.3:180:0.3);

\draw[color=violet] (0.25,-4.55) node[anchor=south] {\circled{\bf 4}};

\draw[thick] (-4,-6) node[anchor=south] {\bf \large \vdots};
\draw[thick] (-3.5,-6) node[anchor=south] {\bf \large etc.};

%
\draw (-3,-6.5) -- (3.5,-6.5);
\foreach \a in {-2.5,-2,...,3}
   \draw[thick] (\a,-6.6) -- (\a,-6.4);

\draw (-2,-6.8) node[anchor=north] {$x$};
\draw (2.5,-6.8) node[anchor=north] {$y$};

\filldraw[black, thick] (0,-6.5) circle (4pt) ;
\filldraw[black, thick] (-1,-6.5) circle (4pt) ;
\filldraw[black, thick] (2.5,-6.5) circle (4pt) ;
\filldraw[black, thick, fill=white] (-2,-6.5) circle (4pt) ;

\draw (5,-6.2) node[anchor=north] {\emph{final configuration}};

\end{tikzpicture}
\caption{Illustration of allowed moves using the mobile cluster, in order to exchange $\eta(x)$ with $\eta(y)$.}
\label{fig:allowedmoves}
\end{figure}
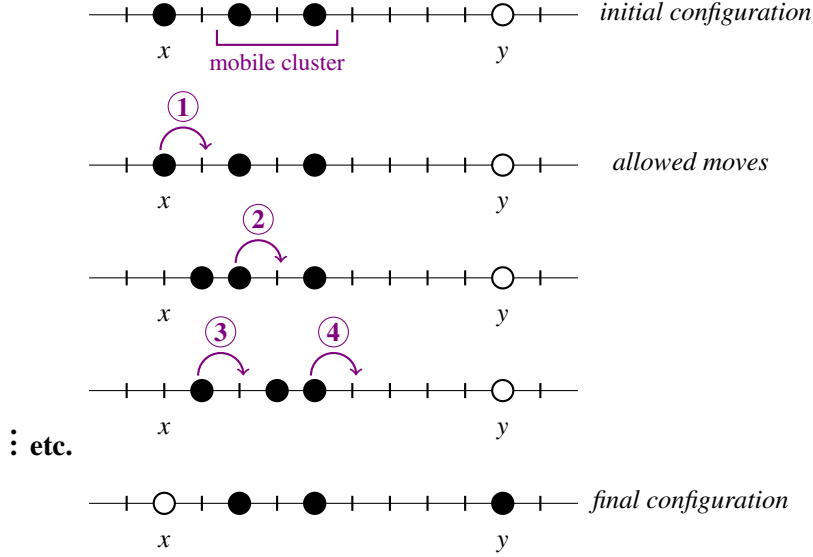

Under $\nu_\rho$, almost surely, there are no blocked configurations. However, given a spatial region of interest, the closest mobile cluster might be very far, and thus virtually useless. In order to control this, we introduce the notion of \emph{good} and \emph{bad} configurations inside a finite box. Fix $x\in\bb Z$, $\ell\in \bb N$ and let $\Lambda_x^\ell:=\{x+1,...,x+\ell\}$ be the box of size $\ell$ to the right of $x$. We define
\begin{align*}
\mc G_{\ell}(x) & := \Big\{ \eta \in \mc E  \; ; \; \sum_{y=x+1}^{x+\ell-2} \big(\eta(y)\eta(y+1) + \eta(y)\eta(y+2)\big)+\eta(x+\ell-1)\eta(x+\ell) >0 \Big\}\\
&\;=\Big\{ \eta\in \mc E \; ; \; \eta \text{ contains a mobile cluster in } \Lambda_x^\ell\Big\},
\end{align*} and therefore its complementary set is 
\[
\mc E \backslash \mc G_{\ell}(x) = \Big\{ \eta \in \mc E\; ; \; \sum_{y=x+1}^{x+\ell-2}  \big(\eta(y)\eta(y+1) + \eta(y)\eta(y+2)\big)+\eta(x+\ell-1)\eta(x+\ell)  =0 \Big\}.
\] 
Note that, under $\nu_\rho$, the probability of the bad set $\mc B_{\ell}(x)$ rapidly decreases to 0 {as $\ell\to \infty$} (see also \cite{GLT}). Indeed, when $\rho \in (0,1)$:
\begin{equation}\label{eq:decaysigma}
\nu_\rho\big({\mc E \backslash \mc G_{\ell}(x)}\big) \le (1-\rho^{2})^{\lfloor\ell/2\rfloor}.
\end{equation}
For general $m$, introducing similar definitions for bad boxes $\mc B^m_{\ell}(x)$, we have
\begin{equation}\label{eq:decaysigmam}
\nu_\rho\big({\mc E \backslash \mc G_{\ell}(x)}\big) \le (1-\rho^{m})^{\lfloor\ell/m\rfloor}.
\end{equation}

\subsection{Density current}

 The generator $\mathcal{L}_n^m$ can be decomposed in the Hilbert space $\mathbb{L}^2(\nu_\rho)$  into its symmetric and antisymmetric parts. More precisely, we write $\mathcal{L}_n^m=\mathcal{S}_n^m+\mathcal{A}_n^m$, where $\mathcal{S}_n^m=(\mathcal{L}_n^m+(\mathcal{L}_n^m)^*)/2$ and $\mathcal{A}_n^m=(\mathcal{L}_n^m-(\mathcal{L}_n^m)^*)/2$, 
with $(\mathcal{L}_n^m)^*$ being the adjoint of $\mathcal{L}_n^m$ in $\mathbb{L}^2(\nu_\rho)$. By the conservation law, for any $x\in\bb Z$, there exists a local function $j^{n,m}_{x,x+1}$ defined on $\mc E$ such that
\begin{equation*}
\mathcal{L}_n^m\eta(x)=j^{n,m}_{x-1,x}(\eta)-j^{n,m}_{x,x+1}(\eta),
\end{equation*}
and $j^{n,m}_{x,x+1}$ is called the instantaneous current of the system at the bond $\{x,x+1\}$. To fix notation we denote \begin{equation*}
\mathcal{S}_n^m\eta(x)=j^{n,m,{\bf s}}_{x-1,x}(\eta)-j^{n,m,{\bf s}}_{x,x+1}(\eta)\quad\textrm{and}\quad \mathcal{A}_n^m\eta(x)=j^{n,m,{\bf a}}_{x-1,x}(\eta)-j^{n,m,{\bf a}}_{x,x+1}(\eta),
\end{equation*}
so that $j^{n,m}_{x,x+1}$ is the sum of the \emph{symmetric current} $j^{n,m,{\bf s}}_{x,x+1}$ and the \emph{antisymmetric current} $j_{x,x+1}^{n,m,{\bf a}}$. 

\begin{remark}\label{rem}
In the definitions above, the currents are defined only up to an additive constant (not depending on the density), which will not matter in the following argument.
\end{remark} 

\begin{proposition}\label{prop:current} 
There exist \begin{itemize}
\item a {local} function $h^m:\mc E\to\mathbb{R}$, 
\item  and $m+1$ {local} functions $\mathbf{P}_k:\mc E \to \bb R$ (which do not depend on $n$), where, for any $k\in\{1,...,m+1\}$, $\mathbf{P}_k$ is a homogeneous multivariate polynomial of degree $k$ in the recentered variables $\{\bar\eta(-m+1),...,\bar\eta(m)\},$
\end{itemize}
such that, for every $x\in\mathbb{Z}$, \begin{align}
\label{grad_cond_sym}
\bar{j}_{x,x+1}^{n,m,\bf s}(\eta)&=\tau_xh^m(\eta)-\tau_{x+1}h^m(\eta),\\ 
\label{decomp_asy}\bar{j}_{x,x+1}^{n,m,\bf a}(\eta)&=n^{-\gamma}\; \sum_{k=1}^{m+1} \tau_x\mathbf{P}_k(\eta).
\end{align}
\end{proposition}

\begin{proof}
  A simple computation shows that the symmetric part of the recentered current $\bar j_{x,x+1}^{n,m,{\bf s}}$ can be written as the gradient of the local function $\tau_xh^m(\eta)$, where
\begin{equation}
h^m(\eta)= \sum_{k=1}^{m} \prod_{j=-(m-k)}^{k-1} \eta(j) - \sum_{k=1}^{m-1} \prod_{\substack{j=-(m-k)\\j\neq 0}}^{k} \eta(j).
\label{eq:h}\end{equation}  
In the case $m=2$, the last formula reduces to
\begin{equation}
h(\eta)=\eta(0)\eta(1)+\eta(0)\eta(-1)-\eta(-1)\eta(1).
\end{equation}
It is also quite easy to see that the antisymmetric part of the current can be {chosen} as
\begin{equation}\label{eq:asymetric}
j_{x,x+1}^{n,m,{\bf a}}(\eta)= \frac{b}{2n^\gamma}\big(\eta(x)+\eta(x+1)-2\eta(x)\eta(x+1)\big)  c_{x,x+1}^m(\eta),
\end{equation} 
where $c_{x,x+1}^m(\eta)$ has been defined in \eqref{eq:rate}. 
Let us now recenter the antisymmetric current by substracting its average. For the sake of simplicity, we write down explicitly the case $m=2$ only, and we let the reader check the general case.
The omitted computations are based on a single formula: for every $(x_1,...,x_k)\in\bb Z^k$, we have
\begin{align}
\eta(x_1)\cdots\eta(x_k) =& \; \bar\eta(x_1)\cdots\bar\eta(x_k) + \rho \sum_{j=1}^k \; \prod_{\substack{ \ell=1 \\ \ell \neq j  }}^k \bar\eta(x_\ell)  + \rho^2 \sum_{\substack{\{i,j\}\\ 1\leq i,j \leq k  }} \; \prod_{\substack{\ell=1 \\ \ell \notin\{i,j\}  }}^k \bar\eta(x_\ell)  \notag \\
& + \cdots + \rho^{k-1} \sum_{j=1}^k \bar\eta(x_j) + \rho^k. \label{eq:centering}
\end{align}
Note that  \eqref{eq:centering} is a sum of homogeneous polynomials  in the variables $\{\bar\eta(x_1),...,\bar\eta(x_k)\}$. The particular case $m=2$ gives
\begin{align}
  \bar j^{n,\bf a}_{x,x+1}(\eta)=& -\frac{b}{2n^\gamma} \bigg\{ 2\bar{\eta}(x-1)\bar{\eta}(x)\bar{\eta}(x+1)+2\bar{\eta}(x)\bar{\eta}(x+1)
\bar{\eta}(x+2)\notag\\
+&4\rho \bar\eta(x)\bar\eta(x+1) +(2\rho-1)(\bar\eta(x-1)+\bar\eta(x+2))(\bar\eta(x+1)+\bar\eta(x))\notag \\
+ & (4\rho^2-2\rho)\big(\bar\eta(x)+\bar\eta(x+1)\big)+
(2\rho^2-2\rho)\big(\bar\eta(x-1)+\bar\eta(x+2)\big)\bigg\}.
\label{intp4}
\end{align}
Then, the proposition follows with 
\begin{align*}
\mathbf{P}_1(\eta)&=-b (2\rho^2-\rho)\big(\bar\eta(0)+\bar\eta(1)\big)-b
(\rho^2-\rho)\big(\bar\eta(-1)+\bar\eta(2)\big) \\
\mathbf{P}_2(\eta)&=- b \; \big\{ 2\rho\bar\eta(0)\bar\eta(1)+(\rho-\tfrac12)\big(\bar\eta(-1)\bar\eta(1)+\bar\eta(0)\bar\eta(2) \\ & \qquad \qquad +\bar\eta(-1)\bar\eta(0)+\bar\eta(1)\bar\eta(2)\big)\big\}\\
\mathbf{P}_3(\eta)&= -b \bar\eta(-1)\bar\eta(0)\bar\eta(1)-b\bar\eta(0)\bar\eta(1)\bar\eta(2).
\end{align*}
\end{proof}

\section{Main results}\label{sec:main_results}
\subsection{Definitions}

From now on, to simplify the notation, we denote by $\eta_{tn^2}$ the random variable $\eta^n_{tn^2}$. 
Our process of interest is the \emph{density fluctuation field}, defined on functions $H$ that belong to the usual Schwartz space $\mathcal{S}(\mathbb{R})$, as
\[\mathcal{Y}^n_t(H):=\frac{1}{\sqrt n}\sum_{x\in\mathbb{Z}}H\Big(\frac{x}{n}\Big)\big(\eta_{tn^2}(x)-\rho\big).\]
 Note that the {expectation of the instantaneous current, in the diffusive time scale, with respect to the equilibrium measure $\nu_\rho$}, equals {(up to an additive constant not depending on $\rho$)} \[\nu_\rho\big(n^2 j_{x,x+1}^{n,m}\big)=bn^{2-\gamma} D(\rho) \chi(\rho).\]  Let us define 
\[
F(\rho):=bD(\rho)\chi(\rho)=bm\rho^m(1-\rho) \quad \text{ and } \quad \mathbf{v}_n:=n^{2-\gamma} \; F'(\rho), \] which describe the \emph{transport behavior of the fluctuations}. {Notice here that the velocity ${\bf v}_n=n^{2-\gamma}\; F'(\rho)$ and the following discussion do not depend on the choice of additive constant mentioned in Remark~\ref{rem}.}  

The fluctuations travel at velocity $\mathbf{v}_n$. Therefore, in order to see a non-trivial evolution of the density fluctuations in the time scale $n^2$, we need to recenter the density fluctuation field by removing $\mathbf{v}_n$ as follows:
\[\widetilde{\mc Y}_t^n(H)= \frac{1}{\sqrt n} \sum_{x\in\bb Z} H\Big(\frac{x-\mathbf{v}_n\; t}{n}\Big) \big(\eta_{tn^2}(x)-\rho\big). \] 
In the following we still denote the field by $\mc Y_t^n$ for simplicity of notation. Actually, a simple computation shows that the choice $\rho=\frac{m}{m+1}$ implies that $\mathbf{v}_n$ vanishes   (see also Section \ref{sec:strategy} for more details on this choice for $\rho$). To lighten further computations we will often assume $\rho=\frac{m}{m+1}$, but we stress that the results hold for any value of $\rho$.

 Now, we need to explain in which sense we obtain an energy solution of the \emph{stochastic Burgers equation}, which is the  macroscopic SPDE satisfied by the limiting density fluctuation field for $\gamma=1/2$. 
 
 \subsection{Stationary energy solutions of the stochastic Burgers equation}
We recall from \cite{GS} the notion of stationary energy solutions of  the SBE given by 
\begin{equation}
\label{eq:sbegene}
d\mc Y_t=\tfrac{1}{2}D(\rho)\Delta \mc Y_t \; dt +\tfrac{1}{2} F''(\rho) \nabla(\mc Y_t^2)\; dt + \sqrt{D(\rho)\chi(\rho)} \; \nabla d\mc W_t.
\end{equation}
Above, $\{\mc W_t \; ; \; t\in [0,T]\}$ is a $\mc S'(\bb R)$-valued Brownian motion, where  $\mathcal{S}'(\bb R)$ is the topological dual of $\mathcal{S}(\bb R)$ with respect to $\bb L^2(\bb R)$.  For a function $H\in\mathcal{S}(\bb R)$, we denote by $\| H\|_2^2$ the usual $\bb L^2(\bb R)$-norm of $H$: $\|H\|_2^2:=\int_\bb R (H(x))^2\, dx.$ Let $\mathcal{C}([0,T],\mathcal{S}'(\bb R))$ be the space of $\mathcal{S}'(\bb R)-$valued continuous trajectories endowed with the weak topology.

\begin{definition}\label{def:sbe}
A process $\{\mathcal{Y}_t\,;\, t\in[0,T]\}$ in $\mathcal{C}([0,T],\mathcal{S}'({\mathbb{R})})$ is a stationary energy solution of \eqref{eq:sbegene} if
\begin{enumerate}
\item for each $t\in [0,T]$, $\mathcal{Y}_t$ is a $\mc S'(\bb R)$-valued white noise with variance $\chi(\rho)$;
\item there exists a positive constant $\kappa$ such that for any $H\in\mathcal{S}(\mathbb{R})$  and $0<\delta<\epsilon<1$
\begin{equation}\label{eq:energy_estimate}
\mathbb{E}\Big[\big(\mathcal{B}_{s,t}^\epsilon(H)-\mathcal{B}_{s,t}^\delta(H)\big)^2\Big]\leq \kappa \epsilon  (t-s)\|\nabla H\|_2^2, \end{equation}
where $$\mathcal{B}_{s,t}^\epsilon(H):=\int_s^t\int_\mathbb{R}\big(\mathcal{Y}_r\ast\iota_\epsilon(x)\big)^2\;\nabla H(x)\,dx\,dr$$
and $\iota_\epsilon(x)$ is an approximation of the identity; 
\item for $H\in\mathcal{S}(\mathbb{R})$, \[\mathcal{Y}_t(H)-\mathcal{Y}_0(H)-\tfrac12 \int_{0}^t\mathcal{Y}_s(D(\rho)\Delta H)\,ds+\tfrac 1 2  F''(\rho)\mathcal B_t(H)\]
is a Brownian motion of variance $D(\rho)\chi(\rho)\|H\|_2^2$,  where $\mathcal{B}_t(H):=\lim_{\epsilon\to0}\mathcal{B}_{0,t}^\epsilon(H)$, the limit being in $\mathbb{L}^2$;
\item the reversed process $\{\mathcal{Y}_{T-t}\,;\, t\in[0,T]\}$ satisfies (3).
\end{enumerate} 
\end{definition}
Above $*$ denotes the convolution operator.

\begin{remark}We note that above we can take, for each $x\in\bb R$ a function  $\iota_\varepsilon(x)(\cdot)=\varepsilon^{-1}\iota(x)(\varepsilon^{-1} \cdot)$ where   $\iota(x)$ is a function in $\bb L^1(\bb R)\cap \bb L^2(\bb R),$ with $\int_{\bb R}\iota(x)(u)du=1$. From, for example, Proposition 2.2 of \cite{GubPer} we know that the process $\mc B_{0,t}^{\varepsilon}$ converges, as $\varepsilon\to 0$ and the limit $\mc B_{0,t}$ does not depend on the choice of  $\iota(x)$.\end{remark}

\subsection{Statement of the convergence result}

From the result in \cite{GubPer}, we have the uniqueness of the stationary controlled solutions of the SBE on the whole line and since the notion of solutions that we use above in Definition \ref{def:sbe} is
equivalent to the one used there (see \cite{GS}, in particular
Proposition 3), we obtain the results stated in the theorem below, which are similar to those of \cite{FGS, gj2014, GJS}. From this we  give a step towards showing the universality of the SBE equation from general microscopic dynamics since our models allow degenerate rates.

 \begin{theorem} \label{theo:conv} Fix $T>0$. The sequence of processes $\{\mc Y_t^n\; ; \; t \in [0,T]\}_{n \in \bb N}$ converges in distribution, as $n\to\infty$, with respect to the Skorokhod topology of $\mc D([0,T]\; ; \; \mc S'(\bb R))$, and moreover  
 \begin{enumerate} 
\item if  $\gamma > \frac12$, it converges to the infinite dimensional Ornstein-Uhlenbeck process solution to
\begin{equation} \label{eq:OU}
d \mc Y_t = \tfrac{1}{2}D(\rho) \Delta \mc Y_t \; dt + \sqrt{D(\rho)\chi(\rho)}\; \nabla d\mc W_t\; ;
\end{equation}

\item if  $\gamma = \frac12$, it converges to the stationary energy solution of the stochastic Burgers equation 
\begin{equation} \label{eq:sbe}
d\mc Y_t=\tfrac{1}{2}D(\rho)\Delta \mc Y_t \; dt +\tfrac{1}{2} F''(\rho) \nabla(\mc Y_t^2)\; dt + \sqrt{D(\rho)\chi(\rho)} \; \nabla d\mc W_t.
\end{equation}
\end{enumerate} 
Above, $\{\mc W_t \; ; \; t\in [0,T]\}$ is a $\mc S'(\bb R)$-valued Brownian motion. 
 \end{theorem}
\begin{remark}
One can compute $ F''(\rho)=bm^2\rho^{m-2} (m-1-(m+1)\rho),$ hence for our simplified choice $m=2$ and $\rho=\frac{2}{3}$ the stochastic Burgers equation \eqref{eq:sbe} becomes 
\begin{equation} \label{eq:sbe-simplified}
d\mc Y_t=\rho \Delta \mc Y_t \; dt -2b \nabla(\mc Y_t^2)\; dt + \sqrt{2\rho\chi(\rho)} \; \nabla d\mc W_t.
\end{equation}
\end{remark}

Thanks to the robust method introduced for the first time in \cite{gj2014}, and used many times since (for example in  \cite{FGS, GJS, GS}),  Theorem \ref{theo:conv} may actually be viewed as a consequence of the  uniqueness result of \cite{GubPer}, together with the BGP2, which is the main result of this paper and is presented in the next section.  Since the general argument to obtain the convergence is now becoming quite standard, we do not want to focus on it. However, for the sake of completeness we expose below the main steps of the proof, and we refer to \cite{FGS, gj2014, GJS, GS} for more details.

\section{Strategy of the proof} \label{sec:strategy} 

In order to simplify the exposition we assume  here $m=2$ and $\rho=\frac{m}{m+1}=\frac{2}{3}$.
The usual structure of the proof is as follows:
\begin{enumerate}
\item first, prove that the sequence of probability measures $\{\mc Q_n\}_{n\in\bb N}$ induced by the energy fluctuation field $\mc Y_t^n$ (which are therefore measures on the Skorokhod space $\mc D([0,T],\mc S'(\bb R))$ is tight;
\item second, write down the approximate martingale problem satisfied by $\mc Y_t^n$ in the large $n$ limit, and prove that it coincides with the martingale characterization of the solutions to the SPDE's given in Theorem \ref{theo:conv}.
\end{enumerate}
The first step follows standard arguments and we refer the interested reader to \cite{GJS}, where this is proved for a perturbation of the models that we consider here. We now focus on the second step, namely the martingale problem obtained in the limit $n\to \infty$. 

 In the case $\gamma = \frac12$, additionally we need to prove that all the conditions stated  in Definition \ref{def:sbe} are satisfied. The martingale formulation given in item (3) is derived below from Dynkin's formula, together with the second order Boltzmann-Gibbs principle and its consequences, as we now describe.  At the end of the section, we merely sketch the derivation of items (1) and (2) since the argument is standard, but for more details we refer the interested reader to \cite{gj2014,GJS,GS}.  Finally, concerning the last item (4), it is enough to consider the reversed dynamics with infinitesimal generator given by $(\mathcal L^m_n)^*$, and repeat the computations that we do in the proof of (3) for the original dynamics.
 
\bigskip
 
\paragraph{\sc Proof of (3):} By Dynkin's formula, for $H\in\mathcal{S}(\mathbb{R})$, 
\begin{equation}
\mc M_t^n(H):=\mathcal{Y}^n_t(H)-\mathcal{Y}^n_0(H)-\int_{0}^{t}n^2\mathcal{L}_n(\mathcal{Y}^n_s(H)) \; ds \label{eq:martingale}
\end{equation}
is a martingale.  Let us define the discrete approximations of the gradient and the Laplacian of $H$ as follows: for any $x\in \bb Z$, 
\begin{align*}
\nabla_nH\Big(\frac{x}{n}\Big):=n\Big[H\Big(\frac{x+1}{n}\Big)-H\Big(\frac{x}{n}\Big)\Big],\quad \Delta_nH \Big(\frac{x}{n}\Big):=n\Big[\nabla_nH\Big(\frac{x}{n}\Big)-\nabla_nH\Big(\frac{x-1}{n}\Big)\Big].
\end{align*}
 A simple computation shows that the integral part of $\mc M_t^n(H)$ can be written as
\begin{equation}\label{eq:integral}
\mathcal{I}_t^n(H):=\int_{0}^{t}\frac{n}{\sqrt n}\sum_{x\in\mathbb{Z}}\nabla_nH\Big(\frac{x}{n}\Big)\; j^{n}_{x,x+1}(\eta_{sn^2})\; ds.
\end{equation}
As a consequence of Proposition \ref{prop:current}, and recalling the explicit computation \eqref{intp4}, one can check that $\mc I_t^n(H)$ defined in \eqref{eq:integral} is the sum of several terms: first, the $h$-\emph{term}
  \begin{equation}
\int_{0}^{t}\frac{1}{2\sqrt n}\sum_{x\in\mathbb{Z}}\Delta_nH\Big(\frac{x}{n}\Big)\tau_xh({\eta}_{sn^2})\;ds,\label{eq:t1}\end{equation}
 plus {a sum of \emph{degree-one-terms}}  {(which correspond to the ${\bf P}_1$-term in Proposition \ref{prop:current})} of the form
 \begin{equation} 
 C(b,\rho) \; \int_{0}^{t}\frac{n^{1-\gamma}}{\sqrt n}\sum_{x\in\mathbb{Z}}\nabla_nH\Big(\frac{x}{n}\Big)\bar{\eta}_{sn^2}{(x+y)}\;ds, \label{eq:t11}
 \end{equation}
 plus \emph{degree-two-terms}  {(which correspond to the ${\bf P}_2$-term)}  of the form 
 \begin{equation}
C(b,\rho)\; \int_{0}^{t}\frac{n^{1-\gamma}}{\sqrt n}\sum_{x\in\mathbb{Z}}\nabla_nH\Big(\frac{x}{n}\Big)\bar{\eta}_{sn^2}({x+y})\bar{\eta}_{sn^2}{(x+z)}\;ds, \label{eq:t2}
\end{equation}
and finally \emph{degree-three-terms} {(which correspond to the ${\bf P}_3$-term)} of the form
 \begin{equation}
C(b,\rho)\; \int_{0}^{t}\frac{n^{1-\gamma}}{\sqrt n}\sum_{x\in\mathbb{Z}}\nabla_nH\Big(\frac{x}{n}\Big)\bar{\eta}_{sn^2}(x)\bar{\eta}_{sn^2}{(x+y)}\bar{\eta}_{sn^2}{(x+z)}\;ds. \label{eq:t3}
\end{equation}
 Above, $y,z$ are fixed integers and $C(b,\rho)$ is a constant that does not depend on $n$. 

Terms \eqref{eq:t1}, \eqref{eq:t11} and \eqref{eq:t3} are not very challenging, and their behavior as $n\to \infty$ is easy to understand. The term \eqref{eq:t2} is more tricky and is investigated by the BGP2, which is proved in Section \ref{sec:BG}.
We end this paragraph by describing the large $n$ limit behavior of all these terms. 

\bigskip

\paragraph{\textbf{a) Limit of the $h$-term}} Note at first that one can replace $\tau_x h$ by $\tau_x h-\nu_\rho(h)$ since the sum $\sum_x \Delta_n H (\frac x n) = 0$. Note then that
 \[
 \tau_xh(\eta)-\nu_\rho(h)-\frac{d}{d\rho}[\nu_\rho(h)]\; \bar\eta(x) =\bar{\eta}(x)\bar{\eta}(x+1)+\bar{\eta}(x)\bar{\eta}(x-1)-\bar{\eta}(x-1)\bar{\eta}(x+1).
 \]
The following result, which is proved in Section \ref{sec:corollary}, Lemma \ref{lemma:1stBG} will help us to conclude:
we have, for any $y\in\mathbb{Z}$, $y\neq 0$ 
\[
\lim_{n\to\infty}\mathbb{E}^n_{{\rho}}\bigg[\Big(\int_{0}^t \frac{1}{\sqrt n}\sum_{x\in\mathbb{Z}}\nabla_n H\Big(\frac x n\Big)\bar{\eta}_{sn^{2}}(x)\bar{\eta}_{sn^{2}}(x+y)ds\Big)^2\bigg]
=0.\]
It follows  that \eqref{eq:t1} can be rewritten as  
 \begin{equation*}
\int_{0}^{t}\frac{1}{2\sqrt n}\sum_{x\in\mathbb{Z}}\Big\{\Delta_nH\Big(\frac{x}{n}\Big)\frac{d}{d\rho}[\nu_\rho(h)]\bar{\eta}_{sn^2}(x)\Big\}ds = \rho \int_0^t \mc Y_s^n(\Delta_n H) \; ds,
\end{equation*}
plus a term which vanishes in $\bb L^2(\bb P_\rho^n)$ as $n\to\infty$. 

\begin{remark}
For $m\geq 3$, $\tau_xh^m(\eta)-\nu_\rho(h^m)-\frac{d}{d\rho}[\nu_\rho(h^m)]\; \bar\eta(x)$ is a sum of centered polynomials of degree $2$ or more. In that case we use Lemmas~\ref{lemma:1stBG},~\ref{lemma:1stBGdeg_three} and Remark~\ref{rem:lemma6.2_higher} to conclude.
\end{remark}

\bigskip

\paragraph{\textbf{b) Limit of the degree-one-terms}} {The sum of all terms of the form \eqref{eq:t11}  will not contribute to the limit}, since we are looking in a reference frame which cancels out any transport behavior (which is equivalent to assuming $\rho=\frac{m}{m+1}$). Only in this paragraph, we can consider any $m\ge 2$ without losing the reader into tedious computations.  Indeed, let us prove the following lemma: 
\begin{lemma} Let $m\geq 2$ and recall the expression for the antisymmetric part of the current given in \eqref{eq:asymetric}. 
If $\rho=\frac{m}{m+1}$, then the degree one polynomial $\mathbf{P}_1$ can be written as the gradient of some {local} function $g^m:\mc E\to \bb R$ which has mean zero with respect to $\nu_\rho$: that is, for every $x\in \bb Z$,
\[\tau_x\mathbf{P}_1(\eta)=\tau_x g^m(\eta)-\tau_{x+1} g^m(\eta).\]
\end{lemma} 
\begin{proof}
Using the formula \eqref{eq:centering}, it is not difficult to figure out and compute the degree one homogeneous polynomial $\mathbf{P}_1$. We let the reader check that 
\begin{multline}
\tau_x\mathbf{P}_1(\eta)={\frac{b}{2}} \bigg\{   m\rho^{m-1}(1-2\rho) \big(\bar\eta(x)+\bar\eta(x+1)\big) \\ + \sum_{k=1}^{m-1} 2(m-k)\rho^{m-1}(1-\rho) \big(\bar\eta(x-k)+\bar\eta(x+1+k)\big)  \bigg\}. \label{eq:deg1} 
\end{multline} Therefore, one can see that, for the special value $\rho=\frac{m}{m+1}$, the degree one part \eqref{eq:deg1} is the gradient of a mean zero function (with respect to $\nu_\rho$), more precisely $\tau_x\mathbf{P}_1$ is equal to:
\[
-{\frac{b}{m}}\; \Big(\frac{m}{m+1}\Big)^{m} \;  \sum_{k=1}^{m-1} k \big(\bar\eta(x)-\bar\eta(x-m+k) + \bar\eta(x+1)-\bar\eta(x+m-k+1)\big). 
\]

\end{proof}
Therefore, the contribution of all degree-one-terms, which is exactly given by 
\[
\int_0^t \frac{n^{1-\gamma}}{\sqrt n} \sum_{x\in\bb Z} \nabla_n H\Big(\frac{x}{n}\Big) \tau_x \mathbf{P}_1(\eta_{sn^2}) \; ds
\]
can be rewritten, after an integration by parts, as
\[
\int_0^t \frac{n^{-\gamma}}{\sqrt n}\sum_{x\in\bb Z} \Delta_nH\Big(\frac{x}{n}\Big) \tau_x g(\eta_{sn^2})\; ds,
\]
which vanishes in $\bb L^2(\bb P_\rho^n)$, as $n\to\infty$, for any $\gamma >0$, from the Cauchy-Schwarz inequality. 

\bigskip

\paragraph{\textbf{c) Limit of degree-two-terms}} Terms of the form \eqref{eq:t2} are more tricky and are treated by the BGP2. For any $x \in \bb Z$ and $\ell \in \bb N$, let us denote by $\vec\eta^\ell(x)$ the empirical density in the box $\Lambda_x^\ell$  defined as
\begin{equation}\label{eq:average}
\vec{\eta}^{\ell}(x)=\frac{1}{\ell}\sum_{y=x+1}^{x+\ell}\bar{\eta}(y).
\end{equation}
For the sake of simplicity we write $\varepsilon n$ for $\lfloor \varepsilon n \rfloor$, its integer part.  With these notations, Theorem \ref{theo:BG} and Corollary \ref{cor_BG} (which are stated in the next section) roughly say that any term of the form \eqref{eq:t2} can be rewritten as 
\begin{equation}\label{eq:rest}
\frac{bn^{1-\gamma}}{\sqrt{n}}\int_0^t\sum_{x\in\bb Z} \nabla_n H\Big(\frac{x}{n}\Big)\Big\{ \big(\vec\eta_{sn^2}^{\varepsilon n}(x)\big)^2 - \frac{\chi(\rho)}{\varepsilon n} \Big\} \; ds 
\end{equation}
plus a term which vanishes in $\bb L^2(\bb P_\rho^n)$ as $n\to\infty$ first, and then $\varepsilon \to 0$. 
Two cases are distinguished:
\begin{enumerate}
\item if $\gamma >\frac12$, the Cauchy-Schwarz inequality proves that \eqref{eq:rest} vanishes in $\bb L^2(\bb P_\rho^n)$: 
\begin{align}
\frac{b^2n^{2-2\gamma}}{n}&\bb E_\rho^n\bigg[\bigg(\int_0^t \sum_{x\in\bb Z} \nabla_n H\Big(\frac{x}{n}\Big) \Big\{ \big(\vec\eta_{sn^2}^{\varepsilon n}(x)\big)^2 - \frac{\chi(\rho)}{\varepsilon n} \Big\} \; ds\bigg)^2\bigg] \notag\\
& \leq Ct^2 n^{1-2\gamma} \bigg\{\varepsilon n \sum_{x\in\bb Z} \Big[  \nabla_n H\Big(\frac{x}{n}\Big)  \Big]^2  \nu_\rho\bigg(\Big( \big(\vec\eta^{\varepsilon n}(0)\big)^2 - \frac{\chi(\rho)}{\varepsilon n}\Big)^2\bigg) \bigg\} \label{eq:ine}\\
& \leq C t^2 n^{1-2\gamma}  (\varepsilon n) \bigg\{\sum_{x\in\bb Z}  \Big[  \nabla_n H\Big(\frac{x}{n}\Big)  \Big]^2\bigg\}  \frac{1}{(\varepsilon n)^2}   = Ct^2 \; \frac{n^{1-2\gamma}}{\varepsilon} \bigg\{\frac1n\sum_{x\in\bb Z}  \Big[  \nabla_n H\Big(\frac{x}{n}\Big)  \Big]^2\bigg\} , \notag
\end{align}
which vanishes as $n\to\infty$ (and $\varepsilon$ fixed), as soon as $\gamma > \frac12$.  {In order to obtain the first inequality \eqref{eq:ine}, we use the fact that $\vec\eta^{\varepsilon n}(x)$ and $\vec\eta^{\varepsilon n}(y)$ are two independent variables if $|x-y| > \varepsilon n$.}  
\item if $\gamma = \frac12$, note that 
\begin{equation}\label{eq:nonlinear}
\int_0^t\sum_{x\in\bb Z} \nabla_n H\Big(\frac{x}{n}\Big)\big(\vec\eta_{sn^2}^{\varepsilon n}(x)\big)^2\; ds = \int_0^t \frac{1}{n}\sum_{x\in\bb Z} \nabla_nH\Big(\frac{x}{n}\Big) \Big[\mc Y_s^n\Big(\iota_\varepsilon \Big(\frac x n\Big)\Big)\Big]^2\; ds,
\end{equation}
where for any $u\in \bb R$, $\iota_\varepsilon(u):\bb R\to\bb R$ is defined as  $\iota_\varepsilon(u)(v):=\varepsilon^{-1}\mathbf{1}_{(u,u+\varepsilon]}(v)$. This last term \eqref{eq:nonlinear} gives rise to the non-linear term appearing in the SBE (see (2) in Definition \ref{def:sbe}).  {Besides, the additional term 
\[ b\int_0^t \sum_{x\in\bb Z} \nabla_n H\Big(\frac{x}{n} \Big) \frac{\chi(\rho)}{\varepsilon n}\; ds\]
equals 0 since the sum of the discrete gradient vanishes. 
}
\end{enumerate}

\bigskip

\paragraph{\textbf{d) Limit of degree-three terms}} Terms of the form \eqref{eq:t3}, for $\gamma \geq \frac{1}{2}$, vanish in $\bb L^2(\bb P_\rho^n)$, as $n\to\infty$. This is a consequence of Lemma \ref{lemma:1stBGdeg_three} (proved in Section \ref{sec:corollary}). 

\bigskip 

Let us now put every contribution together (recall that $m=2$). The martingale decomposition \eqref{eq:martingale} rewrites as 
\begin{align*}
\mc M_t^n(H)  = \mc Y_t^n(H) & -\mc Y_0^n(H) - \rho \int_0^t \mc Y_s^n(\Delta_n H) \; ds \\ & + \mathbf{1}_{\gamma=\frac{1}{2}} \frac{b}{2} \big(-4\rho-4(2\rho-1)\big) \int_0^t \frac{1}{n}\sum_{x\in\bb Z} \nabla_nH\Big(\frac{x}{n}\Big) \Big[\mc Y_s^n\Big(\iota_\varepsilon \Big(\frac x n\Big)\Big)\Big]^2\; ds
\end{align*}
plus a term that vanishes in $\bb L^2(\bb P_\rho^n)$ as $n\to\infty$. We recover exactly the martingale problems satisfied by \eqref{eq:OU} and \eqref{eq:sbe-simplified} as defined previously, in the case $m=2$.  We note however that for other values of $m$ we can follow exactly the same strategy and derive similar results. 

%
%
%
%
%
%
%
\bigskip 

\paragraph{\textsc{Proof of items (1) and (2): }} Let $\mc{Y}_t$ be a limit point of $\mc{Y}_t^n$. First, we fix a test function $H\in\mathcal S(\bb R)$ and a time $t>0$. By computing the characteristic function of $\mathcal{Y}^n_t(H)$, it is simple to show that 
$\mathcal{Y}_t(H)$ is Gaussian with variance 
$\chi(\rho)\|H\|_2^2$. This shows (1). Second, from the estimate that we obtained in the BGP2 (Theorem \ref{theo:BG}), it can be shown, by following the arguments in \cite{gj2014},  that \eqref{eq:energy_estimate} holds, which  implies the existence of $\mathcal B_t$.

\begin{remark}
We note that in the case $m\geq 3$, the strategy of the proof is the same as above. The only difference is that in the decomposition \eqref{eq:t11}-\eqref{eq:t3},  we will have $m+1$ terms  each one corresponding to a  polynomial of degree $k$ with $k\leq m+1$. To treat each one of the terms  with $4\leq k\leq m+1$, we use an ad hoc version of Lemma \ref{lemma:1stBGdeg_three}, please see Remark \ref{rem:lemma6.2_higher}.
\end{remark}

\section{Second order Boltzmann-Gibbs principle}\label{sec:BG} 
\begin{theorem}[Second order Boltzmann-Gibbs principle for degree two polynomial functions]
\label{theo:BG}
For any function $V:\mathbb{Z}\to{\mathbb{R}}$ (possibly depending on $n$) that satisfies: for all $n\in\mathbb{N}$
\begin{equation}\label{vinl2}
 \|V\|_{2,n}^2:=n^{-1}\sum_{x\in\mathbb{Z}}V^2(x)\leq K <\infty,\end{equation}
we have 
\[
\lim_{\varepsilon \to 0} \lim_{n\to\infty}\mathbb{E}^n_{{\rho}}\bigg[\Big(\int_{0}^t \sum_{x\in\mathbb{Z}}V(x)\Big\{\bar{\eta}_{sn^{2}}(x)\bar{\eta}_{sn^{2}}(x+1)-\big(\vec{\eta}_{sn^{2}}^{\varepsilon n}(x)\big)^2+\frac{\chi(\rho)}{\varepsilon n}\Big\}ds\Big)^2\bigg]
=0.\]
\end{theorem}
As a consequence of the previous result,  the  second order Boltzmann-Gibbs principle for occupation sites which are not nearest neighbors can be derived. 

\begin{corollary} \label{cor_BG}
For any $y\in\mathbb{Z}$ and any function $V:\mathbb{Z}\to{\mathbb{R}}$ that satisfies \eqref{vinl2},
\[
\lim_{\varepsilon \to 0} \lim_{n\to\infty}\mathbb{E}^n_{{\rho}}\bigg[\Big(\int_{0}^t \sum_{x\in\mathbb{Z}}V(x)\Big\{\bar{\eta}_{sn^{2}}(x)\bar{\eta}_{sn^{2}}(x+y)-\big(\vec{\eta}_{sn^{2}}^{\varepsilon n}(x)\big)^2+\frac{\chi(\rho)}{\varepsilon n}\Big\}ds\Big)^2\bigg]
=0.\]
\end{corollary}
We do not prove this corollary since one just has to combine Theorem \ref{theo:BG} with the fact that occupation sites can be exchanged, which is a simple application of Lemma~\ref{lem:restriction} and Proposition \ref{prop:one-block} stated below.

\begin{remark}
As explained in the strategy of the proof (Section \ref{sec:strategy}), Theorem \ref{theo:BG} and Corollary \ref{cor_BG} will be used with $V(x)=\nabla_nH(\frac{x}{n})$, and since $H$ is in $\mc{S}(\bb{R})$,  \eqref{vinl2} is indeed satisfied.
\end{remark}

Now, we present the proof of  Theorem \ref{theo:BG}. In fact, we will prove a more precise statement, that is, for any $\ell$ sufficiently large and any $t>0$
\begin{multline}\label{eq:BGexpo}
\mathbb{E}^n_{{\rho}}\bigg[\Big(\int_{0}^t \sum_{x\in\mathbb{Z}}V(x)\Big\{\bar{\eta}_{sn^{2}}(x)\bar{\eta}_{sn^{2}}(x+1)-\big(\vec{\eta}_{sn^{2}}^{\ell}(x)\big)^2+\frac{\chi(\rho)}{\ell}\Big\}ds\Big)^2\bigg]\\
\le  Ct\Big\{ \frac{\ell}{n} +\frac{tn}{\ell^2}\Big\}\|V\|_{2,n}^2.
\end{multline}
We note that by fixing $\varepsilon >0$ and choosing  $\ell=\varepsilon n$, the bound \eqref{eq:BGexpo} implies  Theorem \ref{theo:BG}.

To prove \eqref{eq:BGexpo} we use ideas similar to those exposed in \cite{GS}. Before proceeding, we introduce all the notations we need in this section. For any function $\psi:\mathcal {E} \to\mathbb {R}$ we denote by $\mathrm{Var}(\psi)$ the variance of $\psi$ with respect to the measure $\nu_\rho$. Recall that we denote by $\bar{\eta}(x)=\eta(x)-\rho$ the centered variable and by $\chi(\rho)$ the static compressibility of the system. 
Let us define the Dirichlet form associated with the accelerated process as
\begin{equation}\label{eq:dirichlet}
\mathcal{D}_n(f)=\int_{\mc E} f(\eta) \; (-n^2\mathcal{L}_n)f(\eta) \;\nu_\rho(d\eta)=\frac{n^2}{4}\sum_{x\in\Z}I_{x,x+1}(f),
\end{equation}
where
\[
I_{x,x+1}(f)=\int_{\mc E} c_{x,x+1}(\eta) \; \big(f(\eta^{x,x+1})-f(\eta)\big)^2 \nu_\rho(d\eta).
\] 
 We also define the empirical average directed to the left in the spirit of \eqref{eq:average} as  
 \begin{equation}\label{eq:average_left}
 \vecleft\eta^\ell(x):=\frac{1}{\ell} \sum_{y=x-\ell}^{x-1} \bar\eta(y),\qquad x\in\mathbb{Z}, \ell \in\mathbb{N}.
 \end{equation}
  To prove \eqref{eq:BGexpo}, we introduce first a starting box of size $\ell_0 \in \mathbb{N}$. We are going to reach progressively the box of size $\ell\geq \ell_0$ by  using the convexity inequality $(a+b)^2\le 2a^2+2b^2$ several times, to bound the left-hand side of \eqref{eq:BGexpo} from above (up to a constant) by
\begin{align}
\label{eq:term1}
& \mathbb{E}^n_{{\rho}}\bigg[\Big(\int_{0}^t \sum_{x\in\mathbb{Z}}V(x)\bar{\eta}_{sn^2}(x)\big(\bar{\eta}_{sn^2}(x+1)-\vec{\eta}_{sn^2}^{\ell_0}(x)\big)\;ds\Big)^2\bigg]\\
\label{eq:term2}
& +\ \mathbb{E}^n_{{\rho}}\bigg[\Big(\int_{0}^t \sum_{x\in\mathbb{Z}}V(x)\vec{\eta}_{sn^2}^{\ell_0}(x)\big(\bar {\eta}_{sn^2}(x)-\vecleft{\eta}_{sn^2}^{\ell_0}(x)\big)\;ds\Big)^2\bigg]\\
\label{eq:term3}
& +\ \mathbb{E}^n_{{\rho}}\bigg[\Big(\int_{0}^t \sum_{x\in\mathbb{Z}}V(x)\vecleft\eta_{sn^2}^{\ell_0}(x) \big(\vec\eta_{sn^2}^{\ell_0}(x)-\vec\eta_{sn^2}^\ell(x)\big)\;ds\Big)^2\bigg]
\\ \label{eq:term4}
& +\  \mathbb{E}^n_{{\rho}}\bigg[\Big(\int_{0}^t \sum_{x\in\mathbb{Z}}V(x) \vec\eta_{sn^2}^\ell(x)\big(\vecleft\eta_{sn^2}^{\ell_0}(x)-\bar\eta_{sn^2}(x)\big)\;ds\Big)^2\bigg]\\
\label{eq:term5}
& +\ \mathbb{E}^n_{{\rho}}\bigg[\Big(\int_{0}^t \sum_{x\in\mathbb{Z}}V(x) \Big\{\vec{\eta}_{sn^2}^\ell(x)\bar{\eta}_{sn^2}(x)
-\big(\vec{\eta}_{sn^2}^\ell(x)\big)^2+\frac{\chi(\rho)}{\ell}\Big\}\;ds\Big)^2\bigg].
\end{align}
The decomposition above involves five main terms, that we are going to treat separately. The trickiest term to handle is the third one \eqref{eq:term3}, for which a \emph{multi-scale analysis} is necessary. At the end of the computations, we will make a choice for $\ell_0$, and it will turn out that one can choose $\ell_0=n^\theta$ for any $\theta\in(0,\frac12)$. This choice is going to be more clear in the next paragraphs. 

We cannot directly apply the result of \cite{GS} because of the possible degeneracy of the rates: we first have to restrict the set of configurations on which the integrals given in \eqref{eq:term1}--\eqref{eq:term5} are performed. More precisely, we separate the set of configurations into two sets: the \emph{irreducible} component that contains all configurations with at least one mobile cluster in a suitable position, and the remaining configurations, which have small weight under the equilibrium measure $\nu_\rho$. This is the purpose of Section \ref{sec:restric}. After that, we estimate the contributions of all the terms by using the same ideas as in \cite{GS} (see Sections \ref{sec:multi-scale} and \ref{sec:clever}). This is where the seemingly overly complicated decomposition into five terms comes into play: we need that the difference appearing in each term has a support disjoint from that of the multiplying function (\emph{e.g.} $\bar{\eta}(x)$ and $\bar{\eta}(x+1)-\vec{\eta}^{\ell_0}(x)$ in \eqref{eq:term1}).

 To keep the notation simple in the following argument, we let $C$ and $C(\rho)$  denote  constants that do not depend on $n$ nor on $t,\ell, \ell_0$, that may change from line to line.  In all that follows, $V:\mathbb{Z}\to{\mathbb{R}}$ is a  function that satisfies \eqref{vinl2}.

\subsection{Restriction of the set of configurations} \label{sec:restric}

In the following we denote by $\mathbf{1}_A(\eta)$ the indicator function that equals 1 if $\eta \in A$ and 0 otherwise.
\begin{lemma}\label{lem:restriction}
Let $\psi:\mathcal{E}\to\R$ be a mean zero local function in $\mathbb L^2(\nu_\rho)$ whose  support is contained in $\{-\ell+1,...,\ell\}$.
Then, there exists $C>0$  such that, for any $t>0$, any integer $\ell_0\le \ell$ and any $n \in \bb N$
\begin{multline}
\E_\rho^n\bigg[\Big(\int_0^t \sum_{x\in\bb Z} V(x)\; \tau_x  \psi (\eta_{sn^2})\; ds\Big)^2\bigg] \le Ct^2\; n\ell \; (1-\rho^2)^{\ell_0/2} \; \mathrm{Var}(\psi)\; \|V\|_{2,n}^2\\
+ 2\E_\rho^n\bigg[  \Big(\int_0^t \sum_{x\in\bb Z} V(x)\; \tau_x\psi(\eta_{sn^2})\; \mathbf{1}_{\mc G_{\ell_0}(x+\ell)}(\eta_{sn^2})\; ds\Big)^2\bigg]. \label{eq:restriction}
\end{multline}
\end{lemma}

\begin{remark}\label{rem:restric}
Let us give some highlights for our choice to condition on $\mc G_{\ell_0}(x+\ell)$ in \eqref{eq:restriction}. Fix $x\in\mathbb{Z}$. We impose the configuration $\tau_x\eta_{sn^2}$ to belong to $\mc G_{\ell_0}(x+\ell)$. In other words, we want to find at time $sn^2$ a mobile cluster in the box $\{x+\ell+1,..., x+\ell+\ell_0\}$. Thus, the condition involves a box  of size $\ell_0$ which does not intersect the support of $\tau_x\psi(\eta)$, which is by assumption $\{x-\ell+1,..., x+\ell\}$.  We note that we could have chosen to use the condition $\mathbf{1}_{\mc G_{\ell_0}(x-\ell-\ell_0)}(\eta)$, which asks for a mobile cluster in the box $\{x-\ell-\ell_0+1,...,x-\ell\}$. The result is exactly the same.
\end{remark}

\begin{remark}\label{rem:vanish}
Note that whenever $\ell_0=n^\theta$, with $\theta>0$, and $\ell$ remains polynomial in $n$, the first term on the right-hand side of \eqref{eq:restriction} vanishes as $n\to \infty$.
\end{remark}

\begin{proof}
First, note that, for any $y\in\bb Z$ and $\ell_0 \in \bb N$, $\mathbf{1}_{\mc G_{\ell_0}(y)}+ \mathbf{1}_{\mc E \backslash \mc G_{\ell_0}(y)}\equiv 1$. Therefore, 
since the function $\tau_x\psi$ depends on the configuration variables $\{\eta(x-\ell+1),...,\eta(x+\ell)\}$ we can write $\tau_x\psi$ as a product of two functions with disjoint supports as $\tau_x\psi = \tau_x\psi  \big( \mathbf{1}_{\mc G_{\ell_0}(x+\ell)}+ \mathbf{1}_{\mc E \backslash \mc G_{\ell_0}(x+\ell)}\big).$
Note that $\mathbf{1}_{\mc G_{\ell_0}(x+\ell)}$ is a function that only depends on   \[\big\{\eta(x+\ell+1),\eta(x+\ell+2),...,\eta(x+\ell+\ell_0)\big\}.\] We use an elementary inequality and we get the bound
\begin{align*}
\E_\rho^n\bigg[\Big(\int_0^t \sum_{x\in\bb Z} V(x)\;  \tau_x \psi & (\eta_{sn^2})\; ds\Big)^2\bigg] \\  & \le  2\E_\rho^n\bigg[  \Big(\int_0^t \sum_{x\in\bb Z} V(x)\; \tau_x\psi(\eta_{sn^2})\; \mathbf{1}_{\mc E \backslash \mc G_{\ell_0}(x+\ell)}(\eta_{sn^2})\; ds\Big)^2\bigg]\\
& + 2\E_\rho^n\bigg[  \Big(\int_0^t \sum_{x\in\bb Z} V(x)\; \tau_x\psi(\eta_{sn^2})\; \mathbf{1}_{\mc G_{\ell_0}(x+\ell)}(\eta_{sn^2})\; ds\Big)^2\bigg].
\end{align*}
We now deal with the first term of the right-hand side of last equality. We want to use \eqref{eq:decaysigma}, and for this we decompose the set $\bb Z$ into $2\ell+\ell_0$ infinite disjoint subsets as 
\[
\bb Z = \bigcup_{k=0}^{2\ell+\ell_0-1} \Lambda_k, \quad \Lambda_k:=\Lambda_k(\ell,\ell_0)=k+ (2\ell+\ell_0)\bb Z =\big\{..., k-(2\ell+\ell_0), k, k+(2\ell+\ell_0),...\big\}.
\]
Now, from the convexity inequality $(a_1+\cdots+a_p)^2 \le p (a_1^2+\cdots+a_p^2)$, we can estimate
\begin{align}
\E_\rho^n\bigg[  \Big(\int_0^t &\sum_{x\in\bb Z} V(x)\;  \tau_x\psi(\eta_{sn^2})\; \mathbf{1}_{{\mc E \backslash \mc G_{\ell_0}(x+\ell)}}(\eta_{sn^2})\; ds\Big)^2\bigg] \notag \\
& \le (2\ell+\ell_0) \sum_{k=0}^{2\ell+\ell_0-1} \E_\rho^n\bigg[  \Big(\int_0^t \sum_{x\in\Lambda_k} V(x)\;\tau_x\psi(\eta_{sn^2})\; \mathbf{1}_{{\mc E \backslash \mc G_{\ell_0}(x+\ell)}}(\eta_{sn^2})\; ds\Big)^2\bigg] \notag \\
& \le (2\ell+\ell_0) \; t^2 \; \sum_{k=0}^{2\ell+\ell_0-1} \; \int \Big(\sum_{x\in\Lambda_k} V(x)\;\tau_x\psi(\eta)\; \mathbf{1}_{{\mc E \backslash \mc G_{\ell_0}(x+\ell)}}(\eta)\Big)^2 \nu_\rho(d\eta), \label{eq:lastbound}
\end{align}
where we used the Cauchy-Schwarz inequality and stationarity to get the last line. At this point we note that the sets $\Lambda_k$ introduced above have the following nice property: for any $x,y\in\Lambda_k$ such that $x\neq y$,  $|x-y| \geq 2\ell+\ell_0$, and by the imposed condition on the support of $\psi$, since $\psi$ is centered,
\[\int \tau_x \psi(\eta) \mathbf{1}_{{\mc E \backslash \mc G_{\ell_0}(x+\ell)}}(\eta)\; \tau_y\psi(\eta) \mathbf{1}_{{\mc E \backslash \mc G_{\ell_0}(y+\ell)}}(\eta)\; \nu_\rho(d\eta)=0.\]
Therefore, using the facts that: 1) $\ell_0\le \ell$ and 2) the functions $\tau_x\psi$ and $\mathbf{1}_{{\mc E \backslash \mc G_{\ell_0}(x+\ell)}}$ have disjoint supports, we can bound \eqref{eq:lastbound} by
\begin{multline*}
(2\ell+\ell_0) \; t^2 \; \sum_{k=0}^{2\ell+\ell_0-1} \sum_{x\in\Lambda_k} V^2(x)   \int  \Big(\tau_x\psi(\eta)\; \mathbf{1}_{{\mc E \backslash \mc G_{\ell_0}(x+\ell)}}(\eta)\Big)^2 \nu_\rho(d\eta)\\
 \le C\ell \; t^2 \; \sum_{x\in\bb Z} V^2(x)\mathrm{Var}(\psi) \nu_\rho\big({\mc E \backslash \mc G_{\ell_0}(x+\ell)}\big).
\end{multline*}
 Then, \eqref{eq:restriction} follows from  \eqref{eq:decaysigma}.
\end{proof}

\subsection{Computing and summing all the errors} Let us now explain the strategy of the proof: first, we focus on \eqref{eq:term1}, \eqref{eq:term2} and \eqref{eq:term4}. For these three terms, we apply directly Lemma \ref{lem:restriction} with the appropriate function $\psi$ and bound in each case $\mathrm{Var}(\psi)$ by $1$.  We note, however, that for each term we use a specific function $\psi$ whose support size is always bounded by $\ell$ and in all cases we choose to look at the mobile cluster in a box of  size $\ell_0$. 
 We obtain
\begin{align}
\eqref{eq:term1}&+\eqref{eq:term2}+\eqref{eq:term4}\leq Ct^2n\ell(1-\rho^2)^{\ell_0/2}\|v\|_{2,n}^2 \\
&+
\mathbb{E}^n_{{\rho}}\bigg[\Big(\int_{0}^t \sum_{x\in\mathbb{Z}}V(x)\;\bar{\eta}_{sn^{2}}(x)\big(\bar{\eta}_{sn^{2}}(x+1)-\vec{\eta}_{sn^{2}}^{\ell_0}(x)\big)\; \mathbf{1}_{\mc G_{\ell_0}(x+\ell_0)}(\eta_{sn^2})\; ds\Big)^2\bigg]\label{eq:term1restrict}\\
&+
\mathbb{E}_\rho^n\bigg[\Big(\int_0^t\sum_{x\in\mathbb{Z}} V(x)\; \vec\eta^{\ell_0}_{sn^2}(x) \big(\bar\eta_{sn^2}(x)-\vecleft\eta_{sn^2}^{\ell_0}(x)\big)\; \mathbf{1}_{\mc G_{\ell_0}(x-2\ell_0-1)}(\eta_{sn^2})\; ds\Big)^2\bigg] \label{eq:term2restrict}\\
&+
\mathbb{E}_\rho^n\bigg[\Big(\int_0^t \sum_{x\in\mathbb{Z}} V(x) \; \vec\eta^\ell_{sn^2}(x)\big(\vecleft\eta^{\ell_0}_{sn^2}(x)-\bar\eta_{sn^2}(x)\big)\; \mathbf{1}_{\mc G_{\ell_0}(x-2\ell_0-1)}(\eta_{sn^2})\; ds\Big)^2\bigg]\label{eq:term4restrict}
\end{align}
Concerning \eqref{eq:term3}, the classic strategy (explained in \cite[Section 4.2]{gj2014}) consists in doubling several times the size of the box, starting from $\ell_0$, until reaching the expected final size $\ell$ (Proposition \ref{doub box}). 
Fix $x \in \Z$, assume $\ell=2^M \ell_0$ and $\ell_{k+1}=2\ell_k$. The reader can easily check that
\begin{multline}
\vecleft{\eta}^{\ell_0}(x)\Bigl(\vec{\eta}^{\ell_0}(x)-\vec{\eta}^{\ell_M}(x)\Bigr)=
\sum_{k=0}^{M-1}\vecleft{\eta}^{\ell_k}(x)\Bigl(\vec{\eta}^{\ell_k}(x)-\vec{\eta}^{\ell_{k+1}}(x)\Bigr)\\
+\ \sum_{k=0}^{M-2}\vec{\eta}^{\ell_{k+1}}(x)\Bigl(\vecleft{\eta}^{\ell_k}(x)-\vecleft{\eta}^{\ell_{k+1}}(x)\Bigr)
+\vec{\eta}^{\ell_M}(x)\Bigl(\vecleft{\eta}^{\ell_{M-1}}(x)-\vecleft{\eta}^{\ell_{0}}(x)\Bigr).
\end{multline}
Again, this decomposition might seem overly involved and arbitrary. Let us just say that the constraints that govern this choice are: 1) the decomposition should involve doubling box sizes, 2) each term should be a product of centered functions with disjoint supports, 3) the variances of the multiplicative terms are $1/\ell_k$. These three properties will be crucial in the proof.

By an elementary inequality and using Minkowski's inequality twice, the above decomposition implies that \eqref{eq:term3} is bounded from above by 
\begin{align}
& 3\bigg\{\sum_{k=0}^{M-1} \bigg( \mathbb{E}_\rho^n\Big[\Big(\int_0^t \sum_{x\in\mathbb Z} V(x)\;\vecleft\eta^{\ell_k}_{sn^2}(x)\big(\vec\eta^{\ell_k}_{sn^2}(x)-\vec\eta^{\ell_{k+1}}_{sn^2}(x)\big)\: ds\Big)^2\Big]\bigg)^{1/2}\bigg\}^2 \label{eq:multi_sum1}\\
& +\ 3 \bigg\{\sum_{k=0}^{M-2} \bigg( \mathbb{E}_\rho^n\Big[\Big(\int_0^t \sum_{x\in\mathbb Z} V(x)\;\vec\eta^{\ell_{k+1}}_{sn^2}(x)\big(\vecleft\eta^{\ell_k}_{sn^2}(x)-\vecleft\eta^{\ell_{k+1}}_{sn^2}(x)\big)\: ds\Big)^2\Big]\bigg)^{1/2}\bigg\}^2 \label{eq:multi_sum2}\\
&+\ 3\mathbb{E}_\rho^n\bigg[\Big(\int_0^t \sum_{x\in\bb Z} V(x)\; \vec\eta_{sn^2}^\ell(x) \big(\vecleft\eta^{\ell_{M-1}}_{sn^2}(x)-\vecleft\eta^{\ell_0}_{sn^2}(x)\big)\; ds\Big)^2\bigg].\label{eq:multi_sum3}
\end{align}
As before, we use Lemma~\ref{lem:restriction}  and in all cases we look for the mobile cluster in a box of size $\ell_0$, to bound \eqref{eq:multi_sum1}+\eqref{eq:multi_sum2}+\eqref{eq:multi_sum3} by 
\begin{align}
&3\bigg\{ \sum_{k=0}^{M-1} \bigg(\E_\rho^n\bigg[  \Big(\int_0^t \sum_{x\in\bb Z} V(x)\; \vecleft\eta^{\ell_k}_{sn^2}(x)\big(\vec\eta^{\ell_k}_{sn^2}(x)-\vec\eta^{\ell_{k+1}}_{sn^2}(x)\big)\mathbf{1}_{\mc G_{\ell_0}(x+\ell_{k+1})}(\eta_{sn^2})\; ds\Big)^2\bigg]\bigg)^{1/2}\bigg\}^2\label{eq:lemrestr2}\\
&+3\ \bigg\{\sum_{k=0}^{M-2}\bigg(\E_\rho^n\bigg[  \Big(\int_0^t \sum_{x\in\bb Z} V(x)\; \vec\eta^{\ell_{k+1}}_{sn^2}(x)\big(\vecleft\eta^{\ell_k}_{sn^2}(x)-\vecleft\eta^{\ell_{k+1}}_{sn^2}(x) \big) \notag \\
& \qquad \qquad \qquad \qquad \qquad \qquad \qquad \qquad \quad \times \mathbf{1}_{\mc G_{\ell_0}(x-\ell_{k+1}-\ell_0-1)}(\eta_{sn^2})\; ds\Big)^2\bigg]\bigg)^{1/2}\bigg\}^2\label{eq:lemrestr3}\\
& + 3\mathbb{E}_\rho^n\bigg[\Big(\int_0^t \sum_{x\in\Z}V(x)\vec\eta_{sn^2}^\ell(x)\big(\vecleft\eta_{sn^2}^{\ell_{M-1}}(x)-\vecleft\eta_{sn^2}^{\ell_0}(x)\big)\; \mathbf{1}_{\mc G_{\ell_0}(x-\ell_{M-1}-\ell_0-1)}(\eta_{sn^2}) \;ds\Big)^2\bigg]\label{eq:lemrestr4}\\
&+ Ct^2n\ell(1-\rho^2)^{\ell_0/2}\|V\|^2_{2,n}.
\end{align}
 Putting together all the estimates above, we can bound the left-hand side of \eqref{eq:BGexpo}  by
\[
C t^2\; n\ell\; (1-\rho^2)^{\ell_0/2} \; \|V\|_{2,n}^2
\]
plus the sum of three main parts: first, the sum \eqref{eq:term1restrict}+\eqref{eq:term2restrict}+\eqref{eq:term4restrict}+\eqref{eq:lemrestr4}, which will be completely controlled via the one-block estimate in Section \ref{ssec:one-block}; second, the multi-scale part \eqref{eq:lemrestr2}+\eqref{eq:lemrestr3}, which we bound in Section \ref{sec:multi-scale}, and finally the last term \eqref{eq:term5}, which is investigated in Section \ref{sec:clever}.

\subsection{Path argument}
In this section, we explain how the presence of a good box (guaranteed by the previous section) helps to deal with the possible degeneracy of the rates, with the strategy presented in \cite{GLT}.

\begin{lemma}\label{path}
For any function $f \in {\bb L}^2(\nu_\rho)$, for any integers $\ell_0,\ell,n$, for any $y,z\in\{1,\cdots,\ell\}$, for any mean zero local function $\varphi:\mathcal{E}\to\mathbb{R}$  in $\mathbb L^2(\nu_\rho)$, such that $\varphi(\eta^{y,z})=\varphi(\eta)$ for all $\eta\in\mc E$
\begin{multline}\label{eq:path}
\bigg|2\int\sum_{x\in\Z}V(x)\tau_x\varphi(\eta)\big({\eta}(x+y)-\eta(x+z)\big)\mathbf{1}_{\mc G_{\ell_0}(x+\ell)}(\eta) f(\eta)\nu_\rho(d\eta)\bigg|
\\ \le C\frac{(\ell+\ell_0)^2}{n}\mathrm{Var}(\varphi)\|V\|_{2,n}^2+\mc D_n(f).
 \end{multline}
 Moreover, the same estimate holds when we replace above $\mathbf{1}_{\mc G_{\ell_0}(x+\ell)}(\eta)$ by 
$\mathbf{1}_{\mc G_{\ell_0}(x-\ell_0)}(\eta)$.
\end{lemma}

\begin{proof}
The second part of the lemma follows by a simple space symmetry once we prove \eqref{eq:path}.
We first write the quantity we want to bound as twice its half, and perform the change of variables $\eta\mapsto\eta^{x+y,x+z}$ in one of them. Since $x+y$ and $x+z$ do not belong to the support of  $\mathbf{1}_{\mc G_{\ell_0}(x+\ell) }$ and since $\tau_x\varphi(\eta^{x+y,x+z})=\tau_x\varphi(\eta)$, we obtain that the left-hand side of \eqref{eq:path} is bounded by
\[
\sum_{x\in\Z}\big|V(x){\big| \; \bigg|}  \int\tau_x\varphi(\eta)\mathbf{1}_{\mc G_{\ell_0}(x+\ell)}(\eta)(\eta(x+y)-\eta(x+z))\big[f(\eta)-f(\eta^{x+y,x+z})\bigr]\nu_\rho(d\eta){\bigg|}.
\]
Now, {when $V(x)\neq 0$,} we use the fact that $\eta\in\mc G_{\ell_0}(x+\ell)$. This implies that there exists at least one mobile cluster inside the box $\{x+\ell+1,...,x+\ell+\ell_0\}$, and therefore there exists a sequence of legal moves inside the box $\{x+1,...,x+\ell+\ell_0\}$ that allows to exchange the values $\eta(x+z),\eta(x+y)$, \emph{i.e.\@} to take a particle from the site $x+z$ to the site $x+y$ or vice-versa. We observe that if we had not conditioned on the good set of configurations, since the dynamics is degenerate, we would not have any guarantee that this exchange could be done: if, for example, all sites at distance two or less from $x+z$ are empty, then a particle at $x+z$ has zero rate for jumping. {This kind of \emph{path argument} was already used in \cite{GLT} in order to estimate the Dirichlet form.}

{For $\eta\in\mc G_{\ell_0}(x+\ell)$, let $X(\eta)$ be the position of the first mobile cluster in $\Lambda_{x+\ell}^{\ell_0}$, \emph{i.e.\@} $X(\eta)=\min\big\lbrace x'\in\Lambda_{x+\ell}^{\ell_0}\colon \eta(x')\eta(x'+1)+\eta(x')\eta(x'+2)\geq 1\big\rbrace$. For all configurations such that $X(\eta)=x'$, there exist $N(x')\in\mathbb{N}$ and a sequence $\{x^{(i)}\}_{i=0,\ldots,N(x')}$ which takes values in $\{x+1,...,x+\ell+\ell_0\}$ ,  such that  \[\eta^{(0)}=\eta,\;   \eta^{(i+1)}=\bigl(\eta^{(i)}\bigr)^{x^{(i)},x^{(i)}+1},\]
 and $\eta^{(N(x'))}=\eta^{x+y,x+z},\;$ and $c_{x^{(i)},x^{(i)}+1}(\eta^{(i)})>0$ for any $i\in\{0,\ldots,N(x')\}$. Moreover, we can choose the above sequence in such a way that $N(x')\le C(\ell+\ell_0)$ and the sites $x^{(i)}$ do not take more than six times the same value. In order to construct the sequence (see Figure~\ref{fig:path} below), if for instance $z>y$, bring the mobile cluster to $x+z+1$, then bring the group constituted of the mobile cluster plus the hole/particle initially at $x+z$ to $x+y+1$, exchange the variable at $x+y$ and retrace the steps back.
 
 \begin{figure}
 \begin{center}
 \includegraphics[scale=.4]{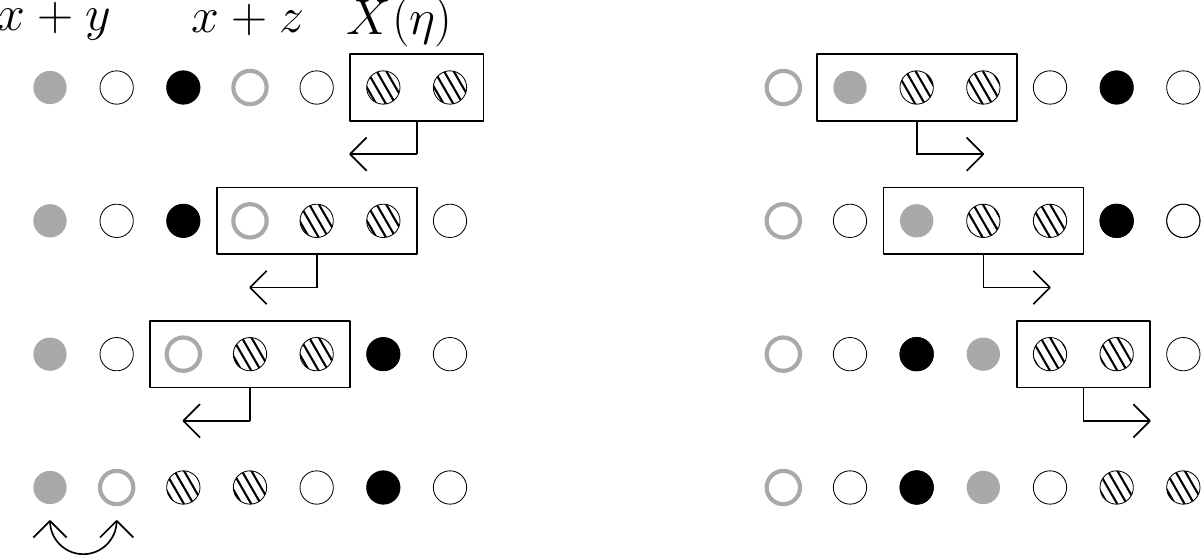}
 \caption{{How to construct the wanted path for the case $\eta(X(\eta))\eta(X(\eta)+1)=1$ (in the other case, start by making the particle at $X(\eta)+2$ jump left). The hole and particle that we want to exchange are in gray, the mobile cluster in stripes and the rest of the configuration in black. By exchanging $X(\eta)-1,X(\eta)$, then $X(\eta),X(\eta)+1$ we move the mobile cluster one step to the left. Similarly, we move the group \{gray hole, striped particles\} to the left until it touches the gray particle. Exchange the gray hole and particle and retrace the steps back.}}
 \label{fig:path}
 \end{center}
 \end{figure}

 Therefore, the quantity \[\int\tau_x\varphi(\eta)({\eta}(x+y)-{\eta}(x+z))\mathbf{1}_{\mc G_{\ell_0}(x+\ell)}(\eta)\;\bigl[ f(\eta)-f(\eta^{x+y,x+z})\bigr]\nu_\rho(d\eta)\] can be rewritten as
\[\sum_{x'\in\Lambda_{x+\ell}^{\ell_0}}\int\tau_x\varphi(\eta)({\eta}(x+y)-{\eta}(x+z))\mathbf{1}_{\{X(\eta)=x'\}}\;\sum_{i=1}^{N(x')}\bigl(f(\eta^{(i-1)})-f(\eta^{(i)})\bigr)\nu_\rho(d\eta).
\]
By Young's inequality, for any $A_x>0$, the last integral is bounded by
\begin{equation}\label{eq:youngineq}
\begin{split}
&\frac{A_x}{2}\sum_{x'\in\Lambda_{x+\ell}^{\ell_0}}\int\sum_{i=1}^{N(x')}(\tau_x\varphi(\eta))^2\frac{({\eta}(x+z)-{\eta}(x+y))^2}{\tau_{x^{(i-1)}}c(\eta^{(i-1)})} \;\mathbf{1}_{X(\eta)=x'} \; \nu_\rho(d\eta)\\
+&\frac{1}{2A_x}\sum_{x'\in\Lambda_{x+\ell}^{\ell_0}}\int\sum_{i=1}^{N(x')}\tau_{x^{(i-1)}}c(\eta^{(i-1)})\mathbf{1}_{\{X(\eta)=x'\}} \big(f(\eta^{(i-1)})-f(\eta^{(i)})\big)^2\;  \nu_\rho(d\eta).
\end{split}
\end{equation}
Note that by definition of the path, for every $i\in\{1,\ldots, N(x')\}$ we have  $ \tau_{x^{(i-1)}}c(\eta^{(i-1)})\ge 1$, so that we can bound the first term by 
\[
CA_x\; (\ell+\ell_0)\int(\tau_x\varphi(\eta))^2{(\bar{\eta}(x+z)-\bar{\eta}(x+y))^2}\; \mathbf{1}_{\mc G_{\ell_0}(x+\ell)}(\eta) \nu_\rho(d\eta)\le C{A_x}\;(\ell+\ell_0) \mathrm{Var}(\varphi).
\]
Since the sequence $\{x^{(i)}\}_i$ takes at most six times the same value, by the  invariance of $\nu_\rho$ under exchanges and because for $X(\eta)=x'$ fixed we construct a deterministic path $(\eta^{(i)})_i$,} we can bound the second term in \eqref{eq:youngineq} by 
\begin{align*}
 \frac{C}{A_x}\int\sum_{k=x+1}^{x+\ell+\ell_0-1}\tau_{k}c(\eta) \big(f(\eta^{k,k+1})-f(\eta)\big)^2\;  \nu_\rho(d\eta)=\frac{C}{A_x}\sum_{k=x+1}^{x+\ell+\ell_0-1} I_{k,k+1}(f).
\end{align*}
Setting $A_x={\big|}V(x){\big|}(\ell+\ell_0)/(2C{n^2})$ {when $V(x)\neq 0$}, we get the desired result.
\end{proof}

\subsection{One-block estimate}
\label{ssec:one-block}

This section aims to treat four terms: \eqref{eq:term1restrict}, \eqref{eq:term2restrict}, \eqref{eq:term4restrict} and \eqref{eq:lemrestr4}. We write the statement and its proof only for \eqref{eq:term1restrict}, and we let the reader adapt the argument to the other cases.

\begin{proposition} [One-block estimate]
\label{prop:one-block} 
Let $\ell_0,\ell\in\bb {N}$ and let $\varphi:\mathcal{E}\to\mathbb{R}$ be a mean zero local function in $\mathbb{L}^2(\nu_\rho)$ whose support does not intersect the set of points $\{1,\dots, \ell\}$. Then for any $t>0$, $n\in\N$ and $y\in\{1,\cdots,\ell\}$:
\begin{multline} \label{eq:lemone}
\E_\rho^n\bigg[  \Big(\int_0^t   \sum_{x\in\bb Z} V(x)\;\tau_x\varphi(\eta_{sn^2})\big(\bar\eta_{sn^2}(x+y)-\vec\eta_{sn^2}^{\ell}(x)\big)\; \mathbf{1}_{\mc G_{\ell_0}(x+\ell)}(\eta_{sn^2})\; ds\Big)^2\bigg] \\ \le
Ct\frac{(\ell+\ell_0)^2}{n} \mathrm{Var}(\varphi)\|V\|_{2,n}^2.
\end{multline}
\end{proposition}

\begin{proof}
By  \cite[Lemma 2.4]{KLO}, we can bound the expectation on the left-hand side of \eqref{eq:lemone} from above by
\[
Ct\sup_{f\in {\bb L}^2(\nu_\rho)}\bigg\{
2\int\sum_{x\in\bb Z}V(x)\tau_x\varphi(\eta)\big(\bar\eta(x+y)-\vec\eta^{\ell}(x)  \big) \mathbf{1}_{\mc G_{\ell_0}(x+\ell)}(\eta) f(\eta)\nu_\rho(d\eta)-\mc D_n(f)\bigg\},
\]
where $\mc D_n(f)$ has been defined in \eqref{eq:dirichlet}. 
Writing \[\bar{\eta}(x+y)-\vec\eta^{\ell}(x)=\frac{1}{\ell}\sum_{z=1}^{\ell}(\eta(x+y)-\eta(x+z))\] and using Lemma~\ref{path}, we get the result.
\end{proof}

\subsection{Estimate of \eqref{eq:lemrestr2}+\eqref{eq:lemrestr3}: multi-scale analysis}
\label{sec:multi-scale}

Now, let us double the size of the box, and consider separately each term in \eqref{eq:lemrestr2} and \eqref{eq:lemrestr3}, for fixed $k\in\{0,...,M-1\}$. Recall that by definition $\ell_k=2^k \ell_0$. 
%
%

\begin{proposition}[Doubling the box]\label{doub box}
For any $t>0$ and $n\in \N$:
\begin{align*}
\mathbb{E}^n_{\rho}\bigg[\Big(\int_{0}^t \sum_{x\in\Z}V(x)\vecleft{\eta}^{\ell_k}_{sn^2}(x)\big(\vec{\eta}_{sn^2}^{\ell_{k}}(x)-\vec{\eta}_{sn^2}^{\ell_{k+1}}(x)\big)\; &  \mathbf{1}_{\mc G_{\ell_0}(x+\ell_{k+1})}(\eta_{sn^2}) \; ds \Big)^2\bigg] \\
  & \le
Ct\; \frac{\ell_k}{n}\|V\|_{2,n}^2,\\
\E_\rho^n\bigg[  \Big(\int_0^t \sum_{x\in\bb Z} V(x)\; \vec\eta^{\ell_{k+1}}_{sn^2}(x)\big(\vecleft\eta^{\ell_k}_{sn^2}(x)-\vecleft\eta^{\ell_{k+1}}_{sn^2}(x) \big)\; & \mathbf{1}_{\mc G_{\ell_0}(x-\ell_{k+1}-\ell_0-1)}(\eta_{sn^2})\; ds\Big)^2\bigg] \\ & \le
Ct\; \frac{\ell_k}{n}\|V\|_{2,n}^2.
\end{align*}
\end{proposition}

\begin{proof}
We present the proof for the first expectation but we note that by symmetry the same arguments applies to the second one.
As above, by \cite[Lemma 2.4]{KLO}  we can bound the first expectation above  by
\begin{equation*}
Ct\sup_{f\in {\bb L}^2(\nu_\rho)}\bigg\{\int \sum_{x\in\Z} V(x)\vecleft{\eta}^{\ell_k}(x)\big(\vec{\eta}^{\ell_{k}}(x)-\vec{\eta}^{\ell_{k+1}}(x)\big)\mathbf{1}_{\mc G_{\ell_0}(x+\ell_{k+1})}(\eta){f(\eta)\nu_\rho(d\eta)}-\mc D_n(f)\bigg\}.
\end{equation*}
 We note that, since $\ell_{k+1}=2\ell_k$ we have
 \begin{equation*}
\vec{\eta}^{\ell_{k}}(x)-\vec{\eta}^{\ell_{k+1}}(x)=\frac{1}{\ell_{k+1}}\sum_{y=x+1}^{x+\ell_{k}}\big(\bar{\eta}(y)-\bar{\eta}(y+\ell_{k})\big),
\end{equation*}
and $\mathrm{Var}\big(\vecleft{\eta}^{\ell_k}(x)\big)=\ell_k^{-1}$. Then we get the result from Lemma~\ref{path}.

\end{proof}

Finally, we show that we can reach the box of size $\ell \geq \ell_0$: if $\ell=2^M\ell_0$, then all we have to do is sum over $k$ (recall \eqref{eq:lemrestr2} and \eqref{eq:lemrestr3}). Otherwise, we choose $M$ big enough so that $2^M\ell_0 \le \ell \le 2^{M+1}\ell_0$. We let the reader check that the error obtained after performing the summation is given by
\[
Ct\; \frac{\ell}{n}\; \|V\|_{2,n}^2.
\]

\subsection{Estimate of \eqref{eq:term5}: clever decomposition}
\label{sec:clever}
To control the last term \eqref{eq:term5} we use an {elementary} inequality to bound it from above by 
\begin{align}\label{eq:last_term}
& 2\mathbb{E}^n_{\rho}\bigg[\Big(\int_{0}^t  \sum_{x\in\mathbb{Z}}V(x)\Big\{\vec{\eta}_{sn^2}^{\ell}(x)\big(\bar{\eta}_{sn^2}(x)-\vec{\eta}_{sn^2}^{\ell}(x)\big)+\frac{\bigl(\bar{\eta}_{sn^2}(x)-\bar{\eta}_{sn^2}(x+1)\bigr)^2}{2\ell}\Big\}ds\Big)^2\bigg]\\
& + 2\mathbb{E}^n_{\rho}\bigg[\Big(\int_{0}^t  \sum_{x\in\mathbb{Z}}V(x)
\Big\{\frac{(\bar{\eta}_{sn^2}(x)-\bar{\eta}_{sn^2}(x+1))^2}{2\ell}-\frac{\chi(\rho)}{\ell}\Big\}ds\Big)^2\bigg] .
\end{align}
It is trivial to check that, by the Cauchy-Schwarz inequality, the last expectation can be bounded from above by \[
 C(\rho) t^2 \frac{n}{\ell^2}\; \|V\|_{2,n}^2.\]
It remains now to bound  \eqref{eq:last_term}.
Let us first write that for any $x\in\mathbb{Z}$ 
\begin{multline}
\vec{\eta}^{\ell}(x)\big(\bar{\eta}(x)-\vec{\eta}^{\ell}(x)\big)+\frac{(\bar{\eta}(x)-\bar{\eta}(x+1))^2}{2\ell}=\frac{1}{\ell}\sum_{y=x+1}^{x+\ell}\vec{\eta}^\ell(x)\bigl(\bar{\eta}(x+1)-\bar{\eta}(y)\bigr)\\+\frac{\ell-1}{\ell}\vec{\eta}^{\ell-1}(x+1)\bigl(\bar{\eta}(x)-\bar{\eta}(x+1)\bigr)+\frac{\bar{\eta}(x)^2-\bar{\eta}(x+1)^2}{2\ell}.
\end{multline}
As before, we check easily that, by Cauchy-Schwarz inequality, 
\begin{equation*}
\mathbb{E}^n_{\rho}\bigg[\Big(\int_{0}^t  \sum_{x\in\mathbb{Z}}V(x)
\Big\{\frac{\big(\bar{\eta}_{sn^2}(x)\big)^2-\big(\bar{\eta}_{sn^2}(x+1)\big)^2}{2\ell}\Big\}ds\Big)^2\bigg] \leq C(\rho) t^2 \frac{n}{\ell^2}\; \|V\|_{2,n}^2.
\end{equation*}
Therefore, by applying Lemma~\ref{lem:restriction}, we can bound \eqref{eq:last_term} by
\begin{align}
& 6\mathbb{E}^n_{\rho}\bigg[\Big(\int_{0}^t  \sum_{x\in\mathbb{Z}}V(x)\frac{1}{\ell}\sum_{y=x+1}^{x+\ell}\Big(\vec{\eta}_{sn^2}^{\ell}(x)\big(\bar{\eta}_{sn^2}(x+1)-\bar{\eta}_{sn^2}(y)\big)\Big) \mathbf{1}_{\mc G_{\ell_0}(x-\ell_0-1)}(\eta_{sn^2}) ds\Big)^2\bigg]\label{eq:clever1}\\
&+6\mathbb{E}^n_{\rho}\bigg[\Big(\int_{0}^t  \sum_{x\in\mathbb{Z}}V(x)\frac{\ell-1}{\ell}\Big(\vec{\eta}_{sn^2}^{\ell-1}(x+1)\big(\bar{\eta}_{sn^2}(x)-\bar{\eta}_{sn^2}(x+1)\big)\Big) \mathbf{1}_{\mc G_{\ell_0}(x-\ell_0-1)}(\eta_{sn^2}) ds\Big)^2\bigg]\label{eq:clever2}\\
&+Ct^2n\ell(1-\rho^2)^{\ell_0/2}\; \|V\|_{2,n}^2+C(\rho) t^2 \frac{n}{\ell^2}\; \|V\|_{2,n}^2. \notag
\end{align}
Now, to bound the first two terms  \eqref{eq:clever1} and \eqref{eq:clever2}, we apply twice \cite[Lemma 2.4]{KLO} to get
\begin{multline*}
Ct\sup_{f\in {\bb L}^2(\nu_\rho)}\bigg\{\int \sum_{x\in\Z} V(x)\frac{1}{\ell}\sum_{y=x+1}^{x+\ell}\vec{\eta}^\ell(x)\bigl(\bar{\eta}(x+1)-\bar{\eta}(y)\bigr)\mathbf{1}_{\mc G_{\ell_0}(x-\ell_0-1)}(\eta)\nu_\rho(d\eta)-\mc D_n(f)\bigg\}\\
+Ct\sup_{f\in {\bb L}^2(\nu_\rho)}\bigg\{\int \sum_{x\in\Z} V(x)\vec{\eta}^{\ell-1}(x+1)\bigl(\bar{\eta}(x)-\bar{\eta}(x+1)\bigr)\mathbf{1}_{\mc G_{\ell_0}(x-\ell_0-1)}(\eta)\nu_\rho(d\eta)-\mc D_n(f)\bigg\}.
\end{multline*}
Now by Lemma~\ref{path}, this is bounded by $Ct\ell \|V\|_{2,n}^2/n$, since 
\[\mathrm{Var}(\vec{\eta}^\ell(x))=\frac{\ell-1}{\ell}\mathrm{Var}(\vec{\eta}^{\ell-1}(x+1))=\frac{1}{\ell}.\]

\subsection{Summing all the errors} 
Putting together every estimate that is obtained for each member of the decomposition \eqref{eq:term1} -- \eqref{eq:term5}, it is quite straightforward to check the following inequality
\begin{multline*}
\mathbb{E}^n_{{\rho}}\bigg[\Big(\int_{0}^t \sum_{x\in\mathbb{Z}}V(x)\Big\{\bar{\eta}_{sn^{2}}(x)\bar{\eta}_{sn^{2}}(x+1)-\big(\vec{\eta}_{sn^{2}}^{\ell}(x)\big)^2+\frac{\chi(\rho)}{\ell}\Big\}ds\Big)^2\bigg]\\
\le  Ct\Big\{tn\; \ell \; (1-\rho^2)^{\ell_0/2}   + \frac{\ell_0^2}{n} + \frac{\ell}{n} +\frac{tn}{\ell^2}\Big\}\|V\|_{2,n}^2.
\end{multline*}
Choosing, for example $\ell_0 = n^{\frac{1}{3}}$ and $\ell=\varepsilon n$, Theorem \ref{theo:BG} follows.

\section{Corollaries of the second order Boltzmann-Gibbs principle}
\label{sec:corollary}
We recall in this section all the lemmas that have been used in Section \ref{sec:strategy}, and we prove them with the same arguments exposed in the proof of the second order Boltzmann-Gibbs principle. 


\begin{lemma} \label{lemma:1stBG}
For any function $V:\mathbb{Z}\to{\mathbb{R}}$ that satisfies
\eqref{vinl2} and any fixed $y\in\mathbb{Z}$,  $y\neq 0$
\[
\lim_{n\to\infty}\mathbb{E}^n_{{\rho}}\bigg[\Big(\int_{0}^t \frac{1}{\sqrt n}\sum_{x\in\mathbb{Z}}V(x)\bar{\eta}_{sn^{2}}(x)\bar{\eta}_{sn^{2}}(x+y)ds\Big)^2\bigg]
=0.\]
\end{lemma}
\begin{proof}
To prove this lemma, we fix $\ell\in\mathbb{N}$ to be chosen ahead. Then, we use an elementary inequality, and we bound the expectation in the statement of the lemma from above by 
\begin{align*}
 &2\mathbb{E}^n_{{\rho}}\bigg[\Big(\int_{0}^t \frac{1}{\sqrt n}\sum_{x\in\mathbb{Z}}V(x)\bar{\eta}_{sn^{2}}(x)\big(\bar{\eta}_{sn^{2}}(x+y)-\vec{\eta}^\ell_{sn^2}(x)\big)\,ds\Big)^2\bigg] \\
 +& 2\mathbb{E}^n_{{\rho}}\bigg[\Big(\int_{0}^t \frac{1}{\sqrt n}\sum_{x\in\mathbb{Z}}V(x)\bar{\eta}_{sn^2}(x)\vec{\eta}^\ell_{sn^{2}}(x)\,ds\Big)^2\bigg]. 
\end{align*}
From Lemma~\ref{lem:restriction} and Proposition \ref{prop:one-block}, for $\ell\geq |y|$ and choosing $\ell_0=\ell$, the first expectation is bounded from above by $C\big(t \frac{\ell^2}{n^2}+t^2\ell(1-\rho^2)^{\ell/2}\big)$. 

 From the Cauchy-Schwarz inequality, since $\nu_\rho\big(\bar{\eta}(x)\vec{\eta}^\ell(x)\bar{\eta}(x')\vec{\eta}^\ell(x')\big)=0$ as soon as $x\neq {x'}$, the  second expectation  is bounded from above by $\frac{t^2}{\ell}.$ Taking, for example, $\ell= n^\theta$ with $\theta<1$  the result follows. 
\end{proof}

\begin{lemma}
\label{lemma:1stBGdeg_three}
For any  function $V:\mathbb{Z}\to{\mathbb{R}}$   that satisfies
\eqref{vinl2}, for any $(y,z)\in\mathbb{Z}^2$, $0<y<z$,
\[
\lim_{n\to\infty}\mathbb{E}^n_{{\rho}}\bigg[\Big(\int_{0}^t \sum_{x\in\mathbb{Z}}V(x)\bar{\eta}_{sn^{2}}(x)\bar{\eta}_{sn^{2}}(x+y)\bar{\eta}_{sn^{2}}(x+z)ds\Big)^2\bigg]
=0.\]
\end{lemma}
\begin{proof}
To prove this lemma, again, we use an {elementary} inequality, and for two integers $\ell\leq L$ we bound the expectation in the statement of the lemma from above by 
\begin{align} \label{eqA}
&C\mathbb{E}^n_{{\rho}}\bigg[\Big(\int_{0}^t \sum_{x\in\mathbb{Z}}V(x)\bar{\eta}_{sn^{2}}(x)\bar{\eta}_{sn^{2}}(x+y)\Big(\bar{\eta}_{sn^{2}}(x+z)-\vec{\eta}^\ell_{sn^2}(x+z)\Big)\,ds\Big)^2\bigg]\\ 
& + C\mathbb{E}^n_{{\rho}}\bigg[\Big(\int_{0}^t \sum_{x\in\mathbb{Z}}V(x)\Big(\bar{\eta}_{sn^{2}}(x)-\bar{\eta}_{sn^2}(x-L)\Big)\bar{\eta}_{sn^{2}}(x+y)\vec{\eta}^\ell_{sn^2}(x+z)\; ds\Big)^2\bigg] \label{eqB} \\
& + C\mathbb{E}^n_{{\rho}}\bigg[\Big(\int_{0}^t \sum_{x\in\mathbb{Z}}V(x)\Big(\bar{\eta}_{sn^{2}}(x-L)-\vecleft{\eta}^L_{sn^2}(x-L)\Big)\bar{\eta}_{sn^{2}}(x+y)\vec{\eta}^\ell_{sn^2}(x+z)\; ds\Big)^2\bigg] \label{eqC} \\
& + C\mathbb{E}^n_{{\rho}}\bigg[\Big(\int_{0}^t \sum_{x\in\mathbb{Z}}V(x)\vecleft{\eta}^L_{sn^2}(x-L)\Big(\bar{\eta}_{sn^{2}}(x+y)-\vecleft{\eta}^L_{sn^2}(x+y)\Big)\vec{\eta}^\ell_{sn^2}(x+z)\; ds\Big)^2\bigg]
\label{eqD} \\
& + C\mathbb{E}^n_{{\rho}}\bigg[\Big(\int_{0}^t \sum_{x\in\mathbb{Z}}V(x)\vecleft{\eta}^L_{sn^2}(x-L)\vecleft{\eta}^L_{sn^2}(x+y)\vec{\eta}^\ell_{sn^2}(x+z)\; ds\Big)^2\bigg]. \label{eqE} \end{align}
From Lemma~\ref{lem:restriction} and Proposition \ref{prop:one-block}, it is easy to bound the first expectation \eqref{eqA}  by $C\big(t\frac{\ell^2}{n}+t^2n\ell(1-\rho^2)^{\ell/2}\big)$. Similarly, {the second term \eqref{eqB} can be treated with Lemma \ref{path} and the exact same ideas as in the proof of Proposition \ref{prop:one-block}}, and one can bound it, as well as the  third one \eqref{eqC}, by $\frac{C}{\ell} \big(t \frac{L^2}{n}+t^2nL(1-\rho^2)^{L/2}\big)$. Finally the fourth one \eqref{eqD} can be bounded by $\frac{C}{\ell L} \big(t\frac{L^2}{n}+t^2nL(1-\rho^2)^{L/2}\big)$. 

By the Cauchy-Schwarz inequality, and using the fact that, for any $x,x'$ such that $|x-x'| \geq L+\ell$, we have
\[
\nu_\rho\big(\vecleft{\eta}^L(x-L)\vecleft{\eta}^L(x+y)\vec{\eta}^\ell(x+z)\vecleft{\eta}^L(x'-L)\vecleft{\eta}^L(x'+y)\vec{\eta}^\ell(x'+z)\big)=0,
\] the last expectation \eqref{eqE} is bounded from above by $\frac{t^2(\ell+L)n}{\ell L^2}.$Taking, for example, $\ell=  \epsilon \sqrt n$ and $L=\epsilon n^{3/4}$ and letting $n\to\infty$, the proof follows. 
\end{proof}

\begin{remark}\label{rem:lemma6.2_higher}
We note that the result of the  previous lemma is true for any polynomial of degree greater or equal to three. The proof follows exactly the same arguments as above. The idea is the following. Suppose that the polynomial has degree $k$ and it is written as $\bar{\eta}_{sn^{2}}(x)\bar{\eta}_{sn^{2}}(x+x_1)\bar{\eta}_{sn^{2}}(x+x_2)\cdots \bar{\eta}_{sn^{2}}(x+x_k)$ and that $0< x_1< x_2< \cdots< x_k$. The first step is to replace the rightmost occupation site, namely, $\bar{\eta}_{sn^{2}}(x+x_k)$ by its occupation average on a box of size $\ell$ to its right. The second step consists in shifting by $L$ the leftmost occupation site, namely $\bar{\eta}_{sn^{2}}(x)$ to its left. Then we replace this shifted occupation site, namely $\bar{\eta}_{sn^{2}}(x-L)$ by its occupation average in a box of size $L$ to its left. Then we replace the occupation site of $\bar{\eta}(x+x_1)$ by its occupation average on a box of size L to its left. The bounds for each one of the aforementioned replacements coincide with those obtained above.
\end{remark}

%
%

\section{Acknowledgments}
This work benefited from the support of the project EDNHS
ANR-14-CE25-0011 of the French National Research Agency (ANR). The work of M.S. was also supported by  CAPES (Brazil) and IMPA (Instituto de Matematica Pura e Aplicada, Rio de Janeiro) through a post-doctoral fellowship, and  in part by the Labex CEMPI (ANR-11-LABX-0007-01). O.B. benefited
from financial support from the seventh Framework Program of the European Union
(7ePC/2007-2013), grant agreement n. 266638. P.G. thanks FCT / Portugal for support through the project UID/MAT/04459/2013.

\end{document}